\documentclass[11pt,a4paper]{article}
\usepackage[a4paper]{geometry}
\usepackage{amssymb,latexsym,amsmath,amsfonts,amsthm}
\usepackage{graphicx}
\usepackage{epsfig}
\usepackage{overpic}
\usepackage{ulem}
\usepackage{comment}
\usepackage[textsize=footnotesize]{todonotes}

\usepackage{tikz}
\usetikzlibrary{decorations}
\usetikzlibrary{arrows}
\usetikzlibrary{shapes}

\usepackage{overpic}

\DeclareMathOperator{\Ai}{Ai}

\DeclareMathOperator{\dist}{dist}
\DeclareMathOperator{\im}{Im}
\DeclareMathOperator{\re}{Re}
\DeclareMathOperator{\sgn}{sgn}

\renewcommand{\Im}{\im}
\renewcommand{\Re}{\re}
\newcommand{\ds}{\displaystyle}

\newtheorem{theorem}{Theorem}[section]
\newtheorem{lemma}[theorem]{Lemma}
\newtheorem{proposition}[theorem]{Proposition}

\theoremstyle{definition}

\newtheorem{rhp}[theorem]{RH problem}

\newtheorem{remark}[theorem]{Remark}

\numberwithin{equation}{section}

\title{Asymptotic behavior and zero distribution of polynomials orthogonal with respect
to Bessel functions}

\author{Alfredo Dea\~no\footnotemark[1]\,, Arno B.J.~Kuijlaars\footnotemark[2]\,, and Pablo Rom\'an\footnotemark[3]}
\date{\today}

\begin{document}

\maketitle
\renewcommand{\thefootnote}{\fnsymbol{footnote}}
\footnotetext[1]{Department of Computer Science, KU Leuven, Celestijnenlaan 200A, 3001 Leuven, 
Belgium, email: alfredo.deano\symbol{'100}cs.kuleuven.be}
\footnotetext[2]{Department of Mathematics, KU Leuven, Celestijnenlaan 200B, 
3001 Leuven, Belgium, email: arno.kuijlaars\symbol{'100}wis.kuleuven.be}
\footnotetext[3]{CIEM, FaMAF, Universidad Nacional de C\'ordoba, Medina Allende s/n 
Ciudad Universitaria, C\'ordoba, Argentina, email: roman\symbol{'100}famaf.unc.edu.ar}

\begin{abstract}
We consider polynomials $P_n$ orthogonal with respect to the weight 
$J_{\nu}$ on $[0,\infty)$, where $J_{\nu}$ is the Bessel function of order $\nu$.
Asheim and Huybrechs considered these polynomials in connection with 
complex Gaussian quadrature for oscillatory integrals. They observed that
the zeros are complex and accumulate as $n \to \infty$ near the 
vertical line $\Re z = \frac{\nu \pi}{2}$. 
We prove this fact for the case $0 \leq \nu \leq 1/2$ from strong
asymptotic formulas that we derive for the polynomials $P_n$ in the complex plane.
Our main tool is the Riemann-Hilbert problem for orthogonal polynomials, 
suitably modified to cover the present situation, and the Deift-Zhou steepest
descent method. A major part of the work is devoted to the construction of 
a local parametrix at the origin, for which we give an existence proof that 
only works for $\nu \leq 1/2$.  
\end{abstract}

\section{Introduction}

In this paper we are interested in the polynomials $P_n$ that are orthogonal with
respect to the weight function $J_{\nu}$ on $[0,\infty)$, where $J_{\nu}$ is
the Bessel function of order $\nu \geq 0$. The Bessel function is oscillatory with an amplitude
that decays like $\mathcal{O}(x^{-1/2})$ as $x \to \infty$, and therefore the moments 
\[ \int_0^{\infty} x^j J_{\nu}(x) dx \] 
do not exist. It follows that the polynomials $P_n$ can not be defined by the usual orthogonality
property 
\begin{equation} \label{Pnx} 
	\int_0^\infty P_n(x) x^j  J_\nu(x) dx =0, \qquad j=0,1,\ldots,n-1.
	\end{equation}

Asheim and Huybrechs \cite{AH} introduced the polynomials $P_n$ via a regularization
of the weight with an exponential factor. For each $s > 0$, they consider the monic polynomial $P_n(x;s)$ 
of degree $n$ that is orthogonal with respect to the weight function $J_{\nu}(x) e^{-sx}$, in the following sense:
\begin{equation} \label{Pnxs} 
	\int_0^\infty  P_n(x;s) x^j J_\nu(x) e^{-sx}dx=0, \qquad j=0,1,\ldots,n-1,
	\end{equation}
and they take the limit
\begin{equation} \label{Pnlimit} 
	P_n(x) = \lim_{s \to 0+} P_n(x; s), 
	\end{equation}
	provided that the limit exists. 
Since the weight function $J_{\nu}(x)e^{-sx}$ changes sign on the positive real axis, there is 
actually no guarantee for existence or uniqueness of $P_n(x;s)$. For the limit \eqref{Pnlimit}
we therefore also have to assume that $P_n(x;s)$ exists and is unique for $n$ large enough. 

The polynomials $P_n$ can alternatively be defined by the moments, since the 
limiting moments for the Bessel function of order $\nu \geq 0$ are known, namely
\begin{equation} \label{moments}
	m_j := 
	\lim_{s \to 0+} \int_0^{\infty} x^j J_{\nu}(x) e^{-sx} dx = 
		2^{j} \frac{\Gamma(\frac{1+\nu+j}{2})}{\Gamma(\frac{1+\nu-j}{2})},
\end{equation}
see \cite[section 3.4]{AH}. Thus we have the determinantal formula (which is familiar from
the general theory of orthogonal polynomials) 
\begin{equation} \label{Pndet} 
	P_n(x) = \frac{1}{\Delta_n} 
	\begin{vmatrix} m_0 & m_1 & \cdots & m_{n-1} & m_n \\ m_1 & m_2  & \cdots & m_{n} & m_{n+1} \\
		\vdots & \vdots & \ddots & \vdots & \vdots \\
		m_{n-1} & m_{n} & \cdots & m_{2n-2} & m_{2n-1} \\
		1 & x & \cdots & x^{n-1} & x^n \end{vmatrix} \end{equation}
with a Hankel determinant $\Delta_n = \det \left[ m_{i+j} \right]_{i,j=0}^{n-1}$.
The polynomial $P_n$ thus exists if and only if $\Delta_n \neq 0$.	

Asheim and Huybrechs \cite{AH} analyze Gaussian quadrature rules with oscillatory  
weight functions, such as complex exponentials, Airy and Bessel functions. The nodes for the Gaussian quadrature
rule are the zeros of the orthogonal polynomials. Since the weight is not real and positive
on the interval of orthogonality there is a problem of existence and uniqueness of the orthogonal
polynomials. In addition, even when the orthogonal polynomial exists, its zeros may not be real, and they
may distribute themselves on some curve or union of curves in the complex plane as the degree tends to infinity.
Examples of this kind of behavior are known in the literature, for instance 
with Laguerre or Jacobi polynomials with non--standard parameters, see \cite{AMMT}, \cite{KuijMcL} and \cite{KuijMF},
and for complex exponentials \cite{Deano}.

\begin{figure}[t]
\centering
\begin{overpic}[width=.45\textwidth]{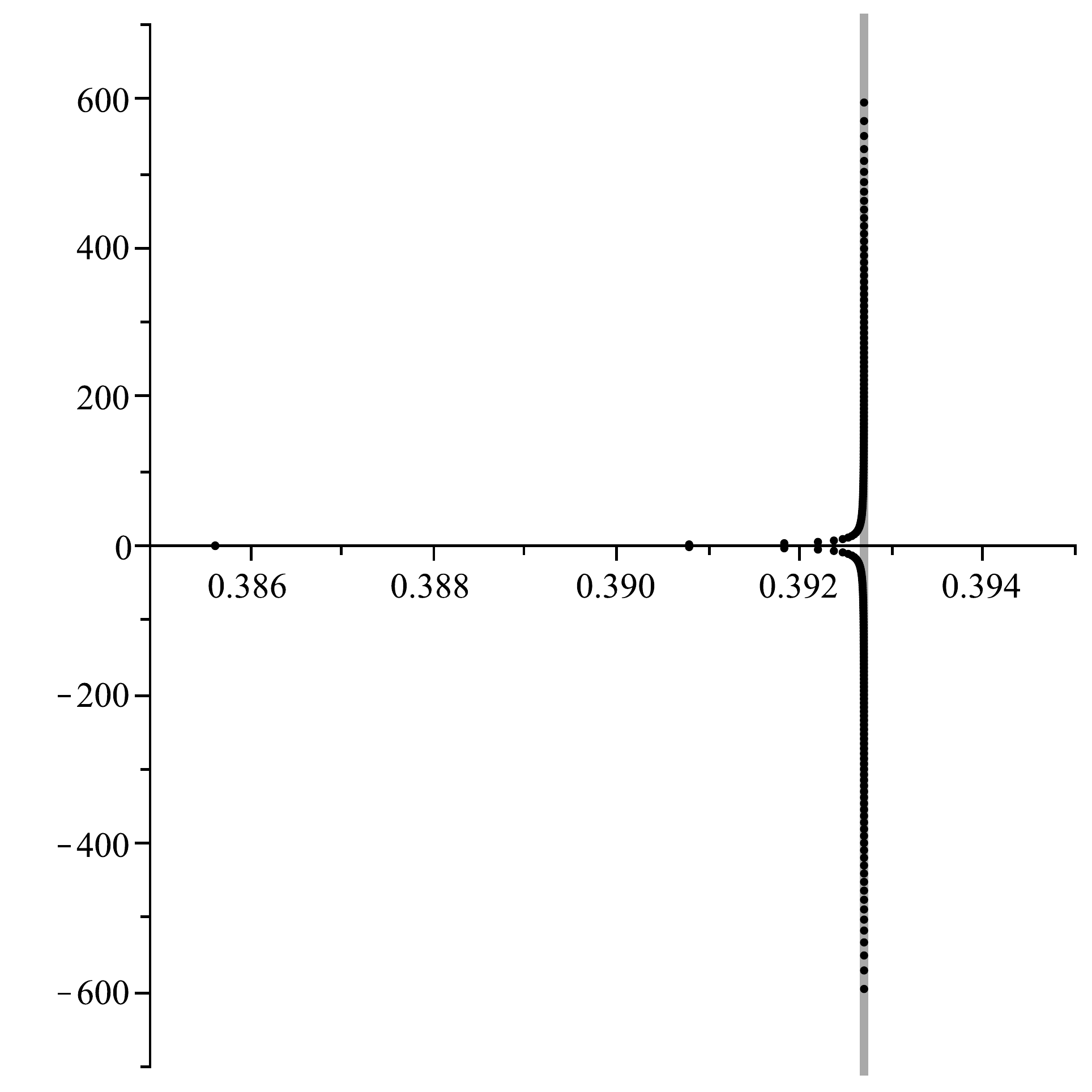}
\end{overpic}
\begin{overpic}[width=.45\textwidth]{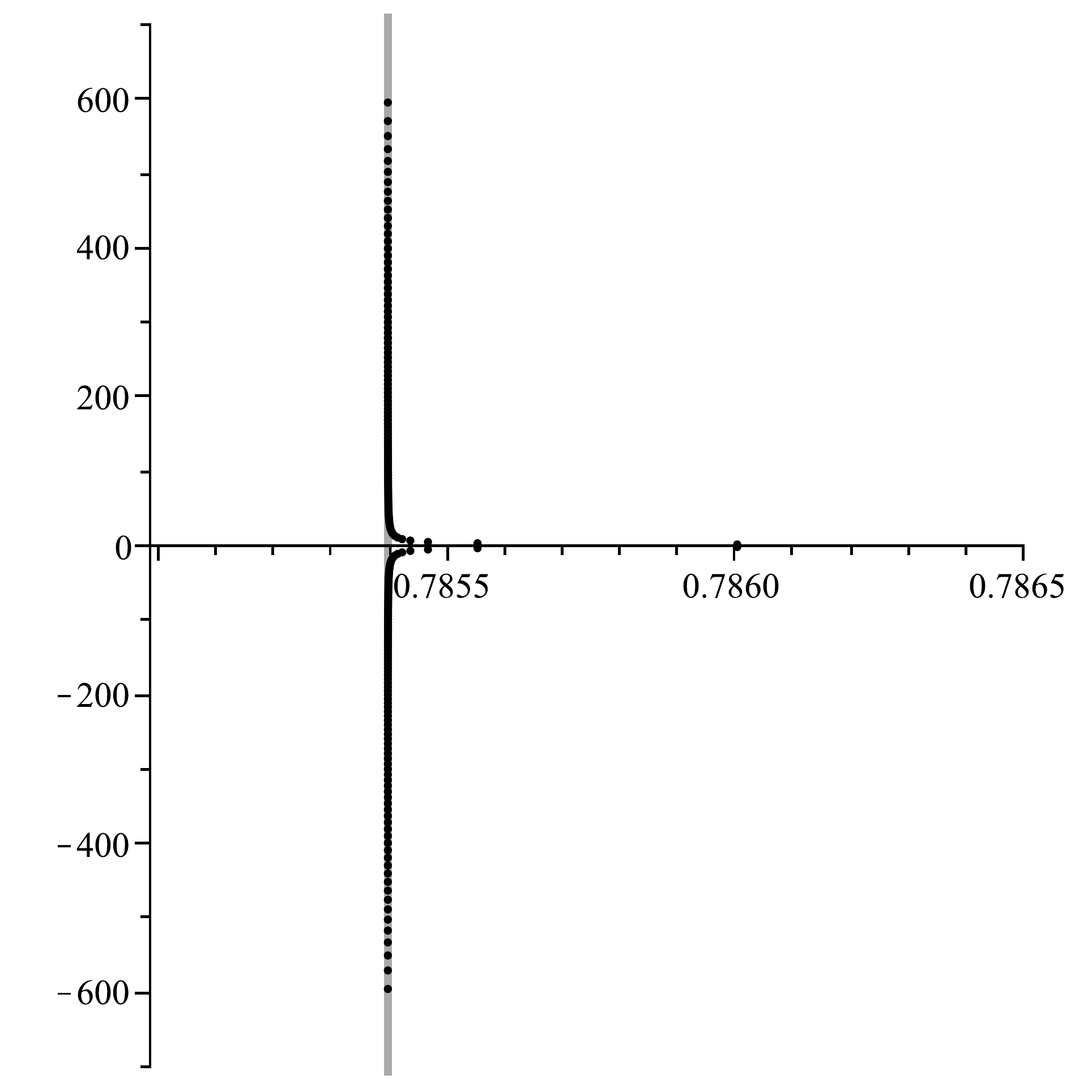}
\end{overpic}
\caption{Plot of the zeros of the polynomials $P_n$ for $n=200$ and $\nu=0.25$ (left), $\nu=0.5$ (right).}
\label{fig:points_plot_small}
\end{figure}
\begin{figure}[t]
\centering
\begin{overpic}[width=.45\textwidth]{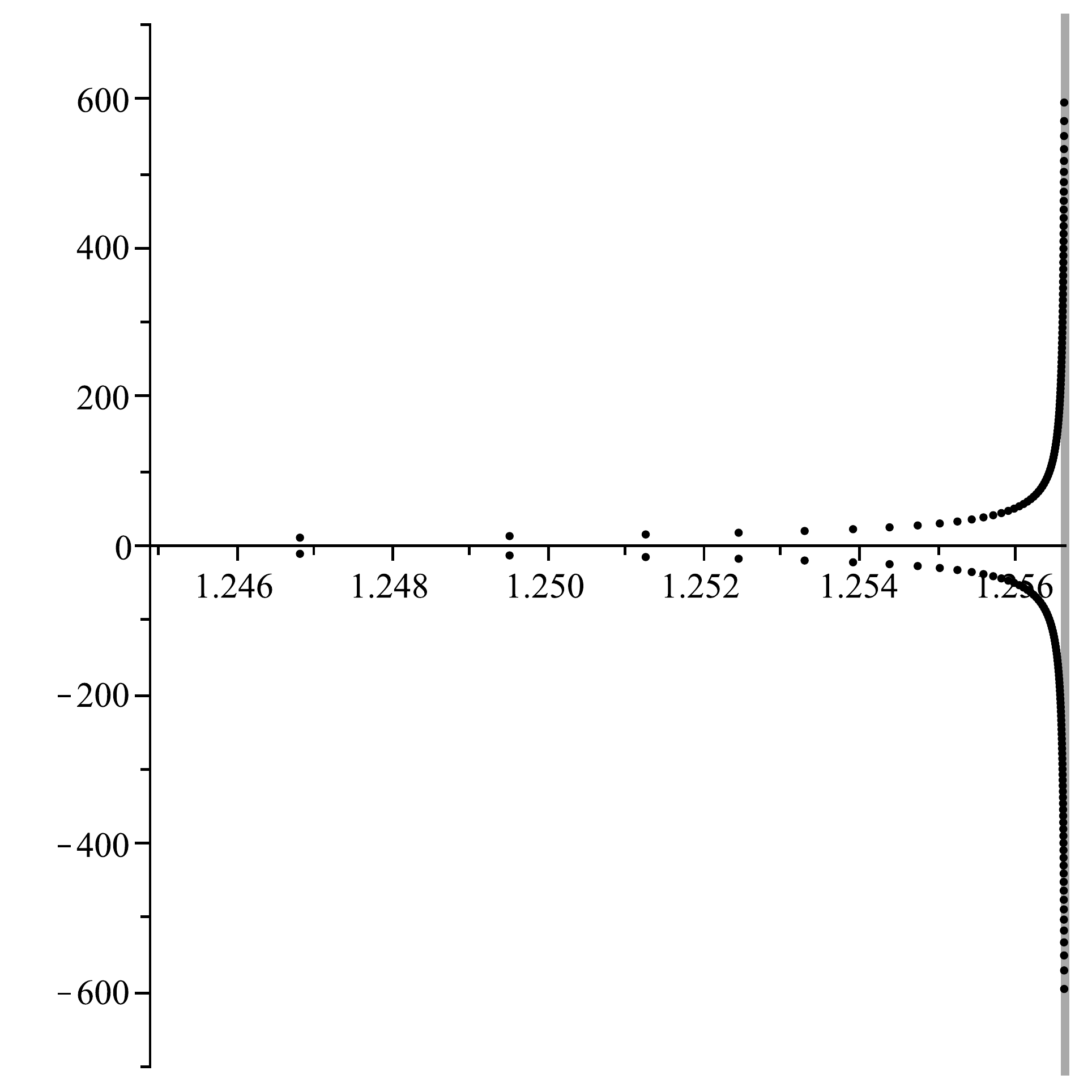}
\end{overpic}
\begin{overpic}[width=.45\textwidth]{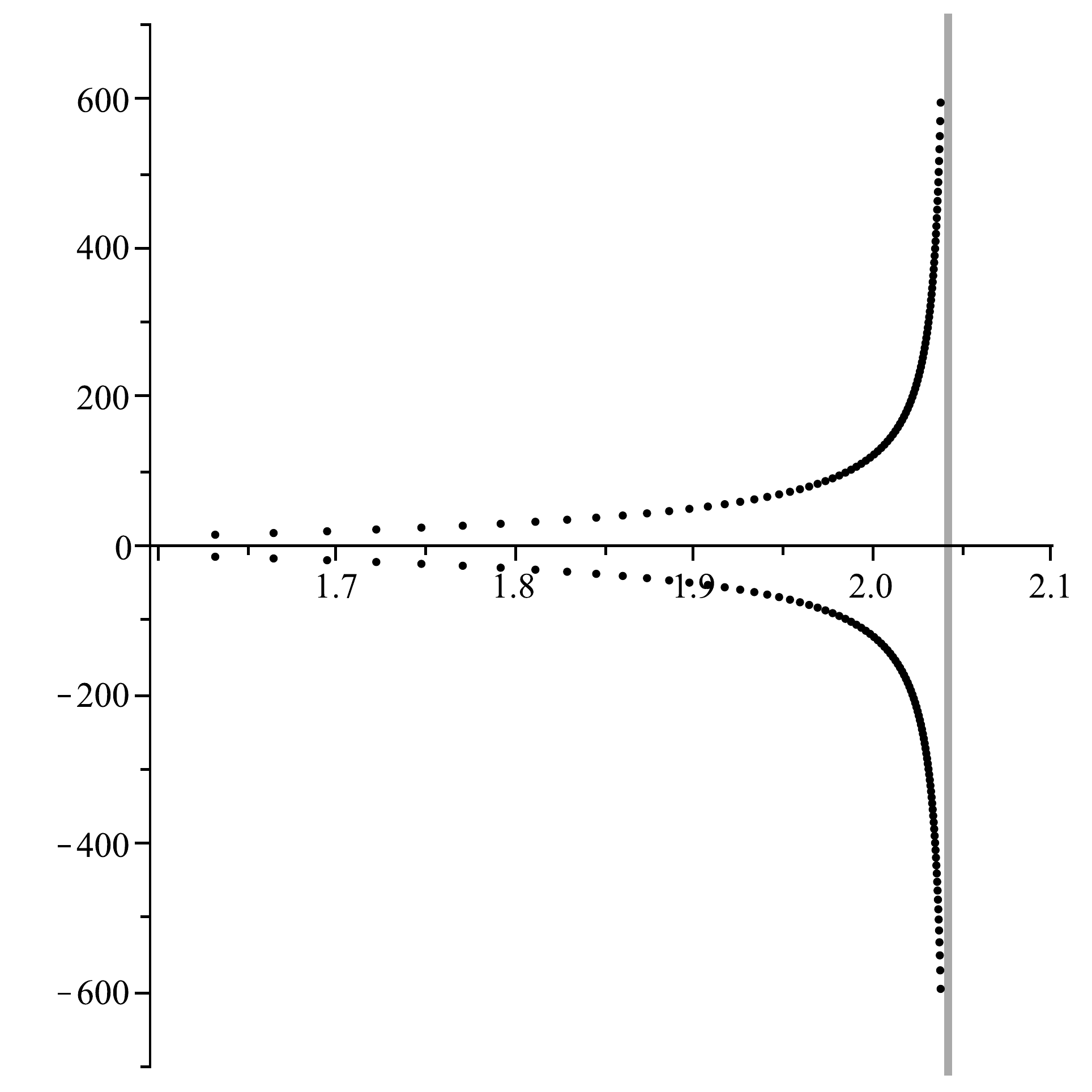}
\end{overpic}
\caption{Plot of the zeros of the polynomials $P_n$ for $n=200$ and $\nu=0.8$ (left), $\nu=1.3$ (right).}
\label{fig:points_plot_big}
\end{figure}

In the present case, with orthogonality defined as \eqref{Pnxs}--\eqref{Pnlimit}, it was shown in 
\cite[Theorem 3.5]{AH} that the zeros of $P_n$ are on the imaginary axis in case $\nu = 0$ and $n$ is even. 
Namely, if $t_1, \ldots, t_{n/2}$ are the zeros of the orthogonal polynomial of degree $n/2$ (where $n$ is even)
with respect to the positive weight $K_0(\sqrt{t}) t^{-1/2}$ on $[0,\infty)$, then the
zeros of $P_n$ are $\pm i \sqrt{t_1}, \ldots \pm i \sqrt{t_{n/2}}$. Here $K_0$ is the modified
Bessel function of the second kind.

For $\nu > 0$ the zeros of $P_n$ are not on the imaginary axis, as is clear from the illustrations given in
\cite{AH}, see also the Figures \ref{fig:points_plot_small} and \ref{fig:points_plot_big}.
The computations have been carried out in {\sc Maple}, using extended precision. 
From these numerical experiments Asheim and Huybrechs \cite{AH} concluded that the zeros seem to cluster
along the vertical line $\Re z = \frac{\nu \pi}{2}$. More precisely, for $\nu \leq \frac{1}{2}$, one sees in
Figure \ref{fig:points_plot_small} that the vast majority of zeros are near a vertical line,
which is indeed at $\Re z = \frac{\nu \pi}{2}$. 

For $\nu > \frac{1}{2}$ one sees in Figure \ref{fig:points_plot_big} that the zeros with large imaginary
part are close to the vertical line $\Re z = \frac{\nu \pi}{2}$, although they are not as close to the vertical
line as the zeros in Figure \ref{fig:points_plot_small}.

We were intrigued by these figures and the aim of this paper is to give a partial explanation
of the observed behavior of zeros. We are able to analyze the polynomials $P_n$ when
$0 \leq \nu \leq \frac{1}{2}$ in the large $n$ limit by means of a Riemann-Hilbert analysis. 
The result is that we indeed find that the real parts of most of the zeros tend to $\frac{\nu \pi}{2}$ as $n \to \infty$.

We are not able to handle the case $\nu > \frac{1}{2}$, since in this case our method to construct
a local parametrix at the origin fails. This difficulty may very well be related to the 
different behavior of the zeros in the case $\nu > \frac{1}{2}$. It would be very interesting to analyze
this case as well. From the figures it seems that there is a limiting curve for the scaled zeros,
if we divide the imaginary parts of the zeros by $n$ and keeping the real parts fixed. This limiting
curve is a vertical line segment if $\nu \leq \frac{1}{2}$ (this will follow from our results), but we do
not know the nature of this curve if $\nu > \frac{1}{2}$.  

\section{Statement of main results}

\subsection{Convergence of zeros}
Our first result is about the weak limit of zeros.
\begin{theorem} \label{Th0} Let $0 < \nu \leq \frac{1}{2}$. Then 
the polynomials $P_n$ exist for $n$ large enough. In addition, the  zeros of $P_n(in \pi z)$
all tend to the interval $[-1,1]$ and have the limiting density
\begin{equation} \label{psidensity2}
	\psi(x)=\frac{1}{\pi}\log \frac{1+\sqrt{1-x^2}}{|x|}, \qquad x\in[-1,1].
\end{equation}
\end{theorem}

The convergence of zeros to the limiting density \eqref{psidensity2} is in the sense of
weak convergence of normalized zero counting measures. 
This means that if $z_{1,n}, \ldots, z_{n,n}$ denote the $n$ zeros of $P_n$, then
\[ \lim_{n \to \infty} \frac{1}{n} \sum_{j=1}^n \delta_{\frac{z_{j,n}}{i \pi n}} = \psi(x) dx \]
in the sense of weak$^*$ convergence of probability measures. Equivalently, we have
\[ \lim_{n \to \infty} \frac{1}{n} \sum_{j=1}^n f\left( \frac{z_{j,n}}{i \pi n}\right) = \int_{-1}^1 f(x) \psi(x) dx \]
for every function $f$ that is defined and continuous in a neighborhood of $[-1,1]$
in the complex plane.

The weak limit of zeros, if we rescale them by a factor $i \pi n$, exists
and does not depend on the value of $\nu$.
Theorem \ref{Th0}  is known to hold for $\nu=0$, and we believe that 
it also holds true for $\nu > \frac{1}{2}$.

Regarding the real parts of the zeros of $P_n$ as $n\to\infty$,  we have
the following result.

\begin{theorem}\label{Th2}
Let $0<\nu\leq 1/2$, and let $\delta>0$ be fixed. Then there exist $n_0\in\mathbb{N}$ 
and $C > 0$ such that for $n\geq n_0$, every zero $z_{j,n}$ of $P_n$ outside the disks 
$D(0,n\delta)$ and $D(\pm n\pi i, n \delta)$ satisfies 
 \begin{equation} \label{Rezjn}
 \left|	\Re z_{j,n} - \frac{\nu\pi}{2} \right| \leq C \epsilon_n,
 \end{equation}
where
\begin{equation} \label{epsilonn}
	\epsilon_n = \frac{n^{\nu-1/2}}{(\log n)^{\nu+1/2}}.
	\end{equation}
\end{theorem}

\begin{remark}
For each fixed $\delta > 0$ there are approximately $\varepsilon n$ zeros of $P_n$ in the disks
$D(0,n\delta)$ and $D(\pm n\pi i, n\delta)$ as $n$ is large, where
\[ \varepsilon = \int_{-1}^{-1+\delta/\pi} \psi(x) dx + \int_{-\delta/\pi}^{\delta/\pi} \psi(x) dx 
	+ \int_{1-\delta/\pi}^1 \psi(x) dx. \]
This is a consequence of the weak convergence of zeros, see Theorem \ref{Th0}.

Clearly, $\varepsilon \to 0$ as $\delta \to 0$, and so it follows from Theorem \ref{Th2} by
taking $\delta$ arbitrarily small that for all but $o(n)$ zeros one has that the real part 
tends to $\frac{\nu \pi}{2}$ as $n \to \infty$.
\end{remark}

\begin{remark}
We do not have information about the zeros in the disk  $D(0,n \delta)$. In our Riemann-Hilbert
analysis we prove the existence of a local parametrix around the origin, but we do not have
an explicit construction with special functions. Therefore we cannot analyze the zeros near the
origin. 

On the other hand, we do have potential access to the extreme zeros in the disks $D(\pm n \pi i, n \delta)$
since the asymptotics of the polynomials $P_n(in \pi z)$ is given in terms of Airy functions. 
From the figures it seems that the result \eqref{Rezjn} also holds for the extreme zeros, but we omit 
this asymptotic result in Theorem \ref{Th1}, since it does not follow clearly from the construction of the
local parametrices in this case.
\end{remark}

\subsection{Orthogonality of $P_n(in \pi z)$ and discussion}

Theorems \ref{Th0} and \ref{Th2} follow from strong asymptotic formulas for the
rescaled polynomials 
\begin{equation} \label{tildePn} 
	\widetilde{P}_n(z) =  (in \pi)^{-n} P_n(in\pi z). 
	\end{equation}
These polynomials are orthogonal polynomials on the real line,
but with a complex weight function.
\begin{proposition} \label{prop:Pntildeorthogonal}
Let $0 \leq \nu < 1$. Then the polynomial $\widetilde{P}_n$ is the monic orthogonal polynomial
of degree $n$ for the weight
\begin{equation} \label{eq:weightnu}
	\begin{cases} e^{ \nu \pi i/2} K_{\nu}(-n \pi x),  & \text{ for } x < 0, \\
	 e^{-  \nu \pi i/2} K_{\nu}(n \pi x),  & \text{ for } x > 0, 
	\end{cases}
	\end{equation}
on the real line. That is,
\begin{equation} \label{Pntildeorthogonal} 
	\int_{-\infty}^{\infty}\widetilde{P}_n(z) x^j  e^{- \sgn(x) \nu \pi i/2} K_{\nu}(n \pi |x|) dx = 0, \qquad j =0,1 \ldots, n-1. 
	\end{equation}
\end{proposition}
The function $K_{\nu}$ in \eqref{eq:weightnu} is the modified Bessel function of second kind of order $\nu$.
Proposition \ref{prop:Pntildeorthogonal} is proved in Section \ref{subsec:second}.

Since $K_{\nu}(x) \sim x^{-\nu}$ as $x \to 0$, see for instance \cite[10.30.2]{DLMF}, 
the condition $\nu < 1$ is necessary
for the convergence of the integral \eqref{Pntildeorthogonal} with $j=0$. In case $\nu=0$
then \eqref{eq:weightnu} is the real and positive weight function $K_{0}(n\pi |x|)$. Then $\widetilde{P}_n$
has all its zeros on the real line, and consequently  the zeros of $P_n$ are on the imaginary axis.
This way we recover the result of \cite{AH}.

For $\nu = 1/2$, the modified Bessel function reduces to an elementary function and the 
weight function \eqref{eq:weightnu} is
\begin{equation} \label{weightnu12} 
	\begin{cases} e^{\pi i/4}  (2n |x|)^{-1/2}  e^{-n \pi |x|}, & \quad x < 0, \\
	e^{-\pi i/4}  (2n |x|)^{-1/2} e^{- n \pi |x|},  & \quad x > 0.
	 \end{cases} \end{equation}
The weight \eqref{weightnu12} has three components:
\begin{itemize}
\item An exponential varying weight $e^{-n \pi |x|}$ with a potential function $V(x) = \pi |x|$ that is
convex but non-smooth at the origin.
\item A square root singularity at the origin $|x|^{-1/2}$.
\item A complex phase factor $e^{\pm \pi i/4}$  with a jump discontinuity at the origin.
\end{itemize}

The exponential varying weight determines the limiting density 
\eqref{psidensity2}. Indeed we have that $\psi(x) dx$ is the minimizer of
the logarithmic energy in external field $\pi |x|$ among probability measures
on the real line, see \cite{ST}, and as is well-known, the zeros of the orthogonal polynomials
with varying weight function $e^{-n \pi |x|}$ have $\psi$ as limiting density.  
This continues to be the case for the weights \eqref{eq:weightnu} as is claimed by Theorem \ref{Th0}.
A Riemann-Hilbert analysis for the weight $e^{-n \pi |x|}$, and other Freud weights, is in \cite{KMcL}.

The square root singularity and the jump discontinuity are known as Fisher-Hartwig singularities
in the theory of Toeplitz determinants. There is much recent progress in the understanding of
Toeplitz and Hankel determinants with such singularities \cite{DIK2}. This is also
related to the asymptotics of the corresponding orthogonal polynomials, whose local behavior 
near a Fisher-Hartwig singularity is described with the aid of confluent hypergeometric functions,
see the works of Deift, Its and Krasovsky \cite{DIK, IK} and also \cite{FMFS,FMFS2}.

We are facing the complication that the Fisher-Hartwig singularity is combined with a logarithmic
divergence of the density $\psi$ at the origin, see \eqref{psidensity2}.
In our Riemann-Hilbert analysis we were not able to construct a local parametrix with special functions, 
and we had to resort to an existence proof, where we used ideas from \cite{KMcL} and \cite{BB},
although our proof is at the technical level different from either of these papers.

\subsection{Asymptotic behavior}
 
Away from the region  where the zeros of $P_n(z)$ lie, the asymptotic behavior is governed by the $g$ function 
associated with the limiting density $\psi$, that is, 
\begin{equation} \label{gfunction}
 g(z)=\int_{-1}^1 \log(z-x)\psi(x)dx, 
\end{equation}
where the density $\psi$ is given by \eqref{psidensity2}. 
Then $g$ is defined and analytic for $z\in \mathbb{C} \setminus(-\infty,1]$.

We prove the following asymptotic behavior of $P_n$ in the region away from the zeros. 
We continue to use $\epsilon_n$ as defined in \eqref{epsilonn}.

\begin{theorem}\label{Th1}
 Let $0<\nu\leq 1/2$. Then the polynomial $P_n$ exists and is unique for sufficiently large $n$. 
Moreover, the polynomial $\widetilde{P}_n$ given by \eqref{tildePn} has the following behavior as $n\to\infty$:
\begin{equation}\label{asymp:Pn:outer}
	\widetilde{P}_n(z)=e^{ng(z)}
\left(\frac{z(z+(z^2-1)^{1/2})}{2(z^2-1)}\right)^{1/4}\left(\frac{(z^2-1)^{1/2}-i}{(z^2-1)^{1/2}+i}\right)^{-\nu/4}
\left(1+\mathcal{O}(\epsilon_n)\right),
\end{equation}
uniformly for $z$ in compact subsets of $\mathbb{C}\setminus [-1,1]$. Here the
branch of the function $(z^2-1)^{1/2}$ is taken which is analytic in 
$\mathbb{C}\setminus[-1,1]$ and positive for real $z > 1$.
\end{theorem}
 
In a neighborhood of $(-1,1)$ we find oscillatory behavior of the polynomials $\widetilde{P}_n$ as  $n\to\infty$.
We state the asymptotic formula \eqref{asymp:Pn:inner} for $\Re z \geq 0$ only. There is an analogous formula for $\Re z < 0$. 
This follows from the fact that the polynomial $P_n$ has real coefficients, as all the moments in
the determinantal formula \eqref{Pndet} are real. Thus $P_n(\overline{z}) = \overline{P_n(z)}$,
and so  
\[ \widetilde{P}_n(-\overline{z}) = \overline{\widetilde{P}_n(z)}, \qquad z \in \mathbb C. \]

To describe the behavior near the interval, we need the analytic continuation of the density \eqref{psidensity2},
which we also denote by $\psi$,
\begin{equation} \label{complexpsi} 
	\psi(z) =  \frac{1}{\pi} \log \frac{1+ (1-z^2)^{1/2}}{z},  \qquad  \Re z > 0, 
	\end{equation}
	which is defined and analytic in $\{ z \mid \Re z > 0\} \setminus [1, \infty)$.
For $\Re z > 0$ with  $z \not\in [1, \infty)$ we also define
\begin{equation} \label{defthetan} 
	\theta_n(z) = n \pi \int_z^1 \psi(s) ds + \frac{1}{4} \arccos z - \frac{\pi}{4}.
	\end{equation}

\begin{theorem} \label{Th3} Let $0 < \nu \leq 1/2$.
There is an open neighborhood $E$ of $(-1,1)$ such that for $z \in E \setminus \{0\}$ with $\Re z \geq 0$ we have
\begin{multline} \label{asymp:Pn:inner}
	\widetilde{P}_n(z)=  \frac{z^{1/4} e^{\frac{\nu \pi i}{4}} e^{n \pi z/2}}{2^{1/4} (2e)^n (1-z^2)^{1/4}} 
	\left[\exp\left( \frac{\nu \pi}{2} \psi(z) + i \theta_n(z)\right) 
	\left( 1 + \mathcal{O}\left(\frac{\log n}{n}\right)\right)  \right. \\
	\left.	+ \exp\left( - \frac{\nu \pi}{2} \psi(z) - i \theta_n(z) \right) 
		\left( 1 + \mathcal{O}\left(\frac{\log n}{n}\right)\right)  + \mathcal{O}(\epsilon_n) \right] 
	\end{multline}
	as $n \to \infty$, 
with $\psi$ and $\theta_n$ given by \eqref{complexpsi} and \eqref{defthetan}.
The asymptotic expansion \eqref{asymp:Pn:inner} is uniform for $z \in E$ with $\Re z \geq 0$ and
$|z-1| > \delta$, $|z| > \delta$, for every $\delta > 0$.
\end{theorem}

The two terms $\exp\left( \frac{\nu \pi}{2} \psi(z) + i \theta_n(z)\right)$ and
$\exp\left(- \frac{\nu \pi}{2} \psi(z) - i \theta_n(z)\right)$   in \eqref{asymp:Pn:inner} 
describe the oscillatory behavior near the interval as well as the leading order
behavior of the zeros. Zeros  can only happen when these two terms are of 
comparable absolute value so that cancellations can take place. When $\nu = 0$ this
happens for real $z \in E$. However, for $\nu > 0$ this does not happen for real $z$,
but near the line $\Im z = -\frac{\nu}{2n}$, as we will show in Section \ref{section44}. 
This leads to Theorem \ref{Th2}.

\subsection{Outline of the paper}

The structure of the rest of the paper is as follows. In Section \ref{Section_RH} we 
state the Riemann--Hilbert problem $Y^{(s)}$ for $P_n(x;s)$ with $s > 0$, and we make
an initial transformation
\begin{equation*}
	Y^{(s)} \mapsto X^{(s)}.
	\end{equation*}
In the RH problem for $X^{(s)}$ we can take the limit $s \to 0+$ which leads to a RH problem for $X$,
that characterizes the polynomial $P_n(x)$.	Then we carry out the further transformations
\begin{equation*}
	X \mapsto U \mapsto T \mapsto S \mapsto Q \mapsto R
\end{equation*}
of the Deift--Zhou nonlinear steepest descent method \cite{Deift,DKMVZ}. The step  $X\mapsto U$ is rotation and scaling, 
to translate the problem to the interval $[-1,1]$. This leads to the polynomials $\widetilde{P}_n$
and the proof of Proposition \ref{prop:Pntildeorthogonal}.
 The normalization at $\infty$ in the 
$U\mapsto T$ step is carried out using an equilibrium problem with a Freud weight 
$w(x)=e^{-n V(x)}$, where $V(x)=\pi|x|$ is the pointwise limit as $n\to\infty$ of the varying weight 
\begin{equation*}
	V_n(x)=-\frac{1}{n}\log K_{\nu}(n\pi |x|). 
\end{equation*}

The construction of the global parametrix $N$ on the interval $[-1,1]$ involves two Szeg\H{o} 
functions $D_1(z)$ and $D_2(z)$, that correspond respectively to an  
algebraic singularity of the weight function at the origin and to a complex phase factor. 
The local parametrices near the endpoints $\pm 1$ involve Airy functions, since the density $\psi(x)$ in 
\eqref{psidensity2} behaves like a square root in a neighborhood of these endpoints. 
The main difficulty of the analysis is the construction of a 
local parametrix in a neighborhood of the origin, and the reason is the lack of analyticity of the weight function 
$V_n(x)$ in that neighborhood. In this paper, we reduce the jump matrices in that local analysis to almost constant 
in a disk around $0$ and then use a small norm argument in $L^2\cap L^{\infty}$ to prove existence of a solution to 
this local RH problem. In this respect, the analysis is similar to the one presented by Kriecherbauer and
 McLaughlin in \cite{KMcL}. Also, the same limiting potential $V(x)$ appears in the work of Bleher and Bothner in 
 \cite{BB}. Another example 
of non--analytic weight function was considered in the work of 
Foulqui\'e, Mart\'inez--Finkelshtein and Sousa, see \cite{FMFS} and \cite{FMFS2},
 although in this case the local parametrix at the origin is explicitly given in terms of 
 confluent hypergeometric functions. 

Finally, in Section \ref{proofs} we follow the transformations both outside and inside the lens, but away 
from the origin, to get the asymptotic information about $P_n(z)$ and its zeros. This proves 
Theorem \ref{Th1} and \ref{Th3}. Theorem \ref{Th0} follows from Theorem \ref{Th1} and
Theorem \ref{Th3} is a consequence of  \ref{Th2}.

\section{Riemann--Hilbert problem}\label{Section_RH}

\subsection{RH problem for polynomials $P_n(x;s)$}
We let $\nu > 0$ and $s > 0$. Orthogonal polynomials are characterized by a matrix valued
Riemann-Hilbert problem as was first shown by Fokas, Its, and Kitaev \cite{FIK}, see also \cite{Deift}.
This characterization does not use the fact that the orthogonality weight is non-negative, and it
therefore also applies to oscillatory weights. Thus the polynomial $P_n(x;s)$ satisfying \eqref{Pnxs} 
is characterized by the following Riemann-Hilbert problem:
\begin{rhp}\label{RHforY}
$Y^{(s)} :\mathbb{C}\setminus [0,\infty) \to \mathbb{C}^{2\times 2}$ is a $2 \times 2$ matrix
valued function that satisfies:
\begin{itemize}
\item[1)] $Y^{(s)}$ is analytic in $\mathbb{C}\setminus [0,\infty)$.
\item[2)] $Y^{(s)}$ satisfies the jump condition
\begin{equation*} 
		Y^{(s)}_{+}(x)= Y^{(s)}_{-}(x) \begin{pmatrix} 1 & J_{\nu}(x)e^{-sx} \\ 0 & 1 \end{pmatrix} \quad \text{on } (0,\infty).
		\end{equation*}
\item[3)] As $z \to \infty$,
\begin{equation}\label{asymp:Y}
	Y^{(s)}(z)=(I+\mathcal{O}(1/z))\begin{pmatrix} z^{n} & 0 \\ 0 & z^{-n} \end{pmatrix},
\end{equation}
where $I$ denotes the $2\times 2$ identity matrix.
\item[4)] $Y^{(s)}(z)$ remains bounded as $z \to 0$.
\end{itemize}
\end{rhp}
The polynomial $P_n(x;s)$ exists and is unique if and only if the RH problem has
a unique solution. In that case we have
\begin{equation} \label{Pn-and-Y11} 
	P_n(x;s) = Y^{(s)}_{11}(x). \end{equation}

\subsection{First transformation}
In the first transformation we use the following connection formula between $J_{\nu}$  
and the modified Bessel function $K_{\nu}$ of the second kind:
\begin{equation}\label{connection}
J_\nu(z)=\frac{1}{\pi i}\left(e^{-\frac{\nu \pi i}{2}}K_\nu(-iz)-e^{\frac{\nu \pi i}{2}}K_\nu(iz)\right), \qquad
|\arg z|\leq \frac{\pi}{2},
\end{equation}
see for instance \cite[formula 10.27.9]{DLMF}. Alternatively, the Bessel function can be written in terms of Hankel functions
as in \cite[formula 10.4.4]{DLMF}.

The formula \eqref{connection} leads to the following factorization of the jump matrix:
\begin{equation} \label{factorization}
    \begin{pmatrix} 1 & J_{\nu}(x)e^{-sx} \\ 0 & 1 \end{pmatrix}
= \begin{pmatrix} 1 & -\frac{e^{\frac{\nu \pi i}{2}}}{\pi i}K_\nu(ix)e^{-sx} \\ 0 & 1 \end{pmatrix}
\begin{pmatrix} 1 & \frac{e^{-\frac{\nu \pi i}{2}}}{\pi i}K_\nu(-ix)e^{-sx} \\ 0 & 1 \end{pmatrix}.
\end{equation}

We define the new matrix valued function $X^{(s)}$ by
\begin{equation}\label{Xs}
X^{(s)}(z)=\begin{cases} \begin{pmatrix} 1 & 0 \\ 0 & (\pi i)^{-1} \end{pmatrix}  
		Y^{(s)}(z)\begin{pmatrix} 1 & - e^{-\frac{\nu \pi i}{2}} K_\nu(-iz)e^{-sz} \\ 0 & \pi i \end{pmatrix},\quad
      &\text{if } 0 <\arg z < \frac{\pi}{2}, \\
      \begin{pmatrix} 1 & 0 \\ 0 & (\pi i)^{-1} \end{pmatrix}
			Y^{(s)}(z)\begin{pmatrix} 1 & - e^{\frac{\nu \pi i}{2}} K_\nu(iz)e^{-sz} \\ 0 & \pi i \end{pmatrix},\quad
      &\text{if }  -\frac{\pi}{2} <\arg z < 0, \\
       \begin{pmatrix} 1 & 0 \\ 0 & (\pi i)^{-1} \end{pmatrix}
			Y^{(s)}(z) \begin{pmatrix} 1 & 0 \\ 0 & \pi i \end{pmatrix}, & \text{elsewhere}.
     \end{cases}
\end{equation}
Then $X^{(s)}$ has an analytic continuation across the positive real axis, due to the factorization \eqref{factorization}.
Thus $X^{(s)}$ is defined and analytic in the complex plane except for the imaginary axis, and it satisfies the following RH problem:
\begin{rhp}\label{RHforX0}
\begin{itemize}
\item[1)] $X^{(s)}$ is analytic in $\mathbb{C}\setminus i\mathbb{R}$.
\item[2)] $X^{(s)}$ satisfies the jump condition (the imaginary axis is oriented from bottom to top)
\begin{equation} \label{jumps:X1}
X^{(s)}_{+}(x)=X^{(s)}_{-}(x) 
\begin{cases}
\begin{pmatrix} 1 &  e^{-\frac{\nu \pi i}{2}} K_\nu(-ix) e^{-sx} \\ 0 & 1 \end{pmatrix}, &
     \text{ for } x \in (0,+i\infty),\\
 \begin{pmatrix} 1 &  e^{\frac{\nu \pi i}{2}} K_\nu(ix) e^{-sx} \\ 0 & 1 \end{pmatrix}, &
     \text{ for }x \in (-i\infty,0).
     \end{cases}
\end{equation}
\item[3)] As $z\rightarrow\infty$,
\begin{equation} \label{asymp:X1}
    X^{(s)}(z)=(I+\mathcal{O}(1/z))\begin{pmatrix} z^{n} & 0 \\ 0 & z^{-n} \end{pmatrix}.
    \end{equation}
\item[4)] $X^{(s)}(z)$ remains bounded as $z \to 0$ with $\Re z < 0$, and 
\begin{equation}  \label{near0:X1}
	X^{(s)}(z)=\begin{pmatrix}
	\mathcal{O}(1) & \mathcal{O}(z^{-\nu}) \\
	\mathcal{O}(1) & \mathcal{O}(z^{-\nu}) \end{pmatrix}, \quad \text{ as } z \to 0 \text{ with } \Re z > 0.
\end{equation}
\end{itemize}
\end{rhp}
The asymptotic condition \eqref{asymp:X1} follows from \eqref{asymp:Y}, the definition \eqref{Xs}
and the fact that  
\begin{equation}\label{asympKv}
K_{\nu}(z)=\left(\frac{\pi}{2z}\right)^{1/2}
e^{-z}\left(1+\mathcal{O}(1/z)\right), \qquad  \text{as } z\to\infty, \quad
|\arg z|<\frac{3\pi}{2},
\end{equation}
see \cite[formula 10.40.2]{DLMF}.
The $\mathcal{O}(z^{-\nu})$ terms in \eqref{near0:X1} appear because of the behavior  
\begin{equation} \label{near0Kv} 
	K_{\nu}(z) \sim \frac{\Gamma(\nu)}{2^{1-\nu}} z^{-\nu}
	\end{equation}
as $z \to 0$ for $\nu>0$, see for instance \cite[formula 10.30.2]{DLMF}.
Note that by \eqref{Pn-and-Y11} and \eqref{Xs}
\begin{equation} \label{Pn-and-X11}	
	P_n(x;s) = X^{(s)}_{11}(x). 
\end{equation}

In the RH problem for $X^{(s)}$ we can take $s \to 0+$. Indeed, after setting $s=0$  in 
\eqref{jumps:X1}, the off-diagonal entries in the jump matrices still tend
to $0$ as $|x| \to \infty$ because of \eqref{asympKv}.
We put $s=0$ and we consider the following RH problem.
\begin{rhp} \label{RHforX}
We seek a function $X:\mathbb{C}\setminus i\mathbb{R} \to \mathbb{C}^{2\times 2}$ satisfying:
\begin{itemize}
\item[1)] $X$ is analytic in $\mathbb{C}\setminus i\mathbb{R}$.
\item[2)] $X$ satisfies the jump condition (the imaginary axis is oriented from bottom to top)
\begin{equation*}
X_{+}(x)=X_{-}(x)
\begin{cases}
   \begin{pmatrix} 1 &  e^{-\frac{\nu \pi i}{2}} K_\nu(-ix)  \\ 0 & 1 \end{pmatrix},
    & \text{ for } x \in (0,+i\infty),\\
    \begin{pmatrix} 1 &  e^{\frac{\nu \pi i}{2}} K_\nu(ix)  \\ 0 & 1 \end{pmatrix},
    & \text{ for }x \in (-i\infty,0).
\end{cases}
\end{equation*}
\item[3)] As $z\rightarrow\infty$,
\begin{equation*}
    X(z)=(I+\mathcal{O}(1/z))\begin{pmatrix} z^{n} & 0 \\ 0 & z^{-n} \end{pmatrix}.
    \end{equation*}
\item[4)] $X(z)$ remains bounded as $z \to 0$ with $\Re z < 0$, and 
\begin{equation*} 
	X(z)=\begin{pmatrix}
	\mathcal{O}(1) & \mathcal{O}(z^{-\nu}) \\
	\mathcal{O}(1) & \mathcal{O}(z^{-\nu}) \end{pmatrix}, \quad \text{ as } z \to 0 \text{ with } \Re z > 0.
\end{equation*}
\end{itemize}
\end{rhp}

If there is a unique solution then the $11$-entry is a monic polynomial of degree $n$, say $P_n$, and
\begin{equation} \label{PnX}
	P_n(x) = X_{11}(z) = \lim_{s \to 0+} X^{(s)}_{11}(z) =  \lim_{s\to 0+} P_n(x;s) 
\end{equation}
see \eqref{Pn-and-X11}. Thus  $P_n$ is the polynomial that we are interested in.

\subsection{Second transformation} \label{subsec:second}
We introduce a scaling and rotation $z\mapsto i\pi n z$ and our main interest is in the rescaled
polynomials $P_n(in\pi z)$ whose zeros will accumulate on the interval $[-1,1]$ as $n\to \infty$. More
precisely, we define $U$ as
\begin{equation}\label{U}
    U(z)=\begin{pmatrix} (in\pi)^{-n} & 0 \\ 0 & (in\pi)^{n} \end{pmatrix} X(in\pi z).
\end{equation}

From \eqref{U} and the RH problem \ref{RHforX}, we immediately obtain the following RH problem for $U(z)$:
\begin{rhp}\label{RHforU}
\begin{itemize}
\item[1)] $U$ is analytic in $\mathbb{C}\setminus \mathbb{R}$.
\item[2)] $U$ satisfies the jump condition
\begin{equation*}
U_{+}(x)=U_{-}(x)
\begin{cases}
\begin{pmatrix} 1 & e^{\nu\pi i/2} K_{\nu}(n\pi|x|) \\ 0 & 1 \end{pmatrix},
\quad x \in (-\infty,0), \\
\begin{pmatrix} 1 & e^{-\nu\pi i/2} K_{\nu}(n\pi|x|) \\ 0 & 1 \end{pmatrix},
\quad x \in (0,\infty).
\end{cases}
\end{equation*}
\item[3)] As $z\rightarrow\infty$,
\begin{equation*}
U(z)=(I+\mathcal{O}(1/z))\begin{pmatrix} z^{n} & 0 \\ 0 & z^{-n} \end{pmatrix}.
\end{equation*}
\item[4)] $U(z)$ remains bounded as $z \to 0 $ with $\Im z > 0$, and 
\begin{equation*} 
U(z)= \begin{pmatrix}
\mathcal{O}(1) & \mathcal{O}(z^{-\nu}) \\
\mathcal{O}(1) & \mathcal{O}(z^{-\nu}) \end{pmatrix} \quad \text{ as }  z \to 0 \text{ with } \Im z < 0. 
\end{equation*}
\end{itemize}
\end{rhp}
Note that by \eqref{PnX}, \eqref{U}, and \eqref{tildePn}
\begin{equation} \label{UnX}
	U_{11}(z) = (i n \pi)^{-n} X_{11}(in \pi z) = (in \pi)^{-n} P_{n}(in \pi z) = \widetilde{P}_n(z)
	\end{equation}
which is a monic polynomial of degree $n$. The zeros of $U_{11}(z)$ are obtained from the
zeros of $P_n$ by rotation over $90$ degrees in the clockwise direction and by dividing
by a factor $\pi n$.

We can now prove Proposition \ref{prop:Pntildeorthogonal}.
\begin{proof}[Proof of Proposition \ref{prop:Pntildeorthogonal}]

The RH problem for $U$ is  the RH problem for orthogonal polynomials on the
real line for the varying weight function $e^{\mp \nu\pi i/2} K_{\nu}(n \pi |x|)$ for $x \in \mathbb R^{\pm}$,
see \cite{Deift,DKMVZ,FIK}.
Because of the $e^{\mp \nu\pi i/2}$ factor, the weight function is not real on the real line, 
and it has a singularity at the origin because of the behavior \eqref{near0Kv}  of the $K_{\nu}$ function near $0$.
The singularity is integrable since $\nu < 1$, 
and so $U_{11} = \widetilde{P}_n$ is the monic polynomial of degree $n$ satisfying \eqref{Pntildeorthogonal}.
\end{proof}

\subsection{Equilibrium problem and third transformation}
In order to normalize the RH problem at infinity we make use of an equilibrium problem
with external field $V(x)=\pi|x|$. The equilibrium measure $\mu$ minimizes
the energy functional
\[ I(\mu) = \iint \log \frac{1}{|x-y|} d\mu(x)d\mu(y) + \int \pi |x| d\mu(x) \]
among all probability measures on $\mathbb{R}$. The minimizer is supported
on $[-1,1]$. It is absolutely
continuous with respect to the Lebesgue measure, $d\mu(x)=\psi(x)dx$, and has density 
\begin{equation*}
	\psi(x)=\frac1\pi \int_{|x|}^1 \frac{1}{\sqrt{s^2-x^2}}ds,
\end{equation*}
which corresponds to the case $\beta=1$ in \cite{KMcL}. The integral
can be evaluated explicitly and it gives the formula \eqref{psidensity2}.
Note that $\psi(x)$ grows like a logarithm at $x = 0$.

The $g$ function is defined in \eqref{gfunction}.
The boundary values $g_+(x)$ and $g_-(x)$ on the real axis satisfy
\begin{equation}\label{gpgm}
g_+(x)-g_-(x)=\begin{cases} 2\pi i,\quad &x\leq -1, \\
                                  2\pi i \ds \int_x^1 \psi(s)ds,\quad &-1 < x < 1,\\
                                  0,\quad & x\geq 1.
                    \end{cases}
\end{equation}

The Euler-Lagrange equations for the equilibrium problem imply that we have (see e.g.~\cite{Deift} or \cite{ST})
\begin{equation}\label{var2}
	g_{+}(x)+g_{-}(x)- \pi |x| \begin{cases} = \ell, & \quad x\in[-1,1], \\
		 <\ell, & \quad  x\in(-\infty,-1)\cup(1,\infty). \end{cases}
\end{equation}
with the constant $\ell$  (see Theorem IV.5.1 in \cite{ST} or formula (3.5) in  \cite{KMcL})
\begin{equation} \label{ell} 
	\ell = - 2 - 2 \log 2. 
	\end{equation}

A related function is
\begin{equation}  \label{phifunction} 
	\varphi(z) = g(z) - \frac{V(z)}{2} - \frac{\ell}{2}
\end{equation}
where 
\begin{equation} \label{Vfunction}
V(z)=\begin{cases}
     \pi z, & \qquad \Re z >0,\\
     -\pi z, & \qquad \Re z <0.
     \end{cases}
\end{equation}

The $\varphi$-function is analytic in $\mathbb{C}\setminus\left((-\infty,1]\cup i\mathbb{R}\right)$. 
For $x\in[-1,1]$ we have from the variational equation \eqref{var2}
\begin{equation}\label{phig}
\begin{aligned}
	\varphi_+(x) & = g_+(x)-\frac{V(x)}{2}-\frac{\ell}{2} = \frac{1}{2}(g_+(x)-g_-(x)), \\
	\varphi_-(x) & =-\varphi_+(x).
	\end{aligned}
\end{equation}

Thus $2\varphi$ gives an analytic extension of $g_+(x)-g_-(x)$ from $[-1,1]$ into the upper half
plane minus the imaginary axis, and of $g_-(x) - g_+(x)$ into the lower half plane minus the
imaginary axis. 
Note that $\varphi_{\pm}(x)$ is purely imaginary on $[-1,1]$, because of \eqref{gpgm}.

On the imaginary axis, the function $\varphi(z)$ is not analytic because of the 
discontinuity in $V(z)$. The boundary values of this weight function satisfy
\begin{equation*}
V_-(z)=V_+(z)+2\pi z,
\end{equation*}
and as a consequence, 
\begin{equation*}
\varphi_-(z)=\varphi_+(z)-\pi z, \qquad z \in i \mathbb R.
\end{equation*}
Here we take the orientation of the imaginary axis from bottom to top.

Now we are ready for the third transformation of the RH problem and we define the matrix valued function
\begin{equation} \label{T}
	T(z)=e^{-n\ell\sigma_3/2} (2n)^{\sigma_3/4} U(z)e^{-n(g(z)-\ell/2) \sigma_3} (2n)^{-\sigma_3/4},
\end{equation}
where $\sigma_3=\begin{pmatrix} 1 & 0\\0 &-1\end{pmatrix}$ is the third
Pauli matrix. 
We also write
\begin{equation} \label{eq:definition-W}
 W_n(x)= \sqrt{2n} K_{\nu}(n\pi |x|)e^{n\pi|x|}, \qquad x \in \mathbb R.
\end{equation}
Then from the above definitions and properties and from the RH problem \ref{RHforU} for $U$ we
find that $T$ satisfies the following Riemann--Hilbert problem.

\begin{rhp}\label{RHforT}
\begin{itemize}
\item[1)] $T$ is analytic in $\mathbb{C}\setminus \mathbb{R}$.
\item[2)] $T$ satisfies the jump conditions
\begin{equation*}
T_{+}(x)=T_{-}(x)
\begin{cases}
\begin{pmatrix} 1 & \, e^{\nu\pi i/2} W_n(x) e^{2n \varphi_+(x)} \\ 0 & 1\end{pmatrix},
\quad x\in(-\infty,-1),\\
\begin{pmatrix} e^{-2n\varphi_{+}(x)} & e^{\nu\pi i/2} W_n(x)  \\ 0 & e^{-2n\varphi_{-}(x)}\end{pmatrix},
\quad x\in(-1,0),\\
\begin{pmatrix} e^{-2n\varphi_{+}(x)} & e^{-\nu\pi i/2} W_n(x) \\ 0 & e^{-2n\varphi_{-}(x)}\end{pmatrix},
\quad x\in(0,1),\\
\begin{pmatrix} 1 & e^{-\nu\pi i/2} W_n(x) e^{2n \varphi_+(x)} \\ 0 & 1\end{pmatrix},
\quad x\in(1,\infty), 
\end{cases}
\end{equation*}
where $W_n$ is given in \eqref{eq:definition-W}.

\item[3)] As $z\rightarrow\infty$,
\begin{equation*}
T(z)=I+\mathcal{O}(1/z).
\end{equation*}
\item[4)] $T(z)$ remains bounded as $z \to 0$ with $\Im z > 0$, and 
\begin{equation} \label{at0:Tgeneral}
T(z)=\begin{pmatrix}
\mathcal{O}(1) & \mathcal{O}(z^{-\nu}) \\
\mathcal{O}(1) & \mathcal{O}(z^{-\nu}) \end{pmatrix},\quad \text{ as } z \to 0 \text{ with } \Im z < 0. 
\end{equation}
\end{itemize}
\end{rhp}

The off--diagonal elements in the jump matrices on $(-\infty,-1)$ and $(1,\infty)$ tend to $0$ at an
exponential rate, because of the Euler--Lagrange condition \eqref{var2}. 

\subsection{Fourth transformation}

The jump matrix on the interval $(-1,0)$ has a factorization
\begin{multline*}
\begin{pmatrix} e^{-2n\varphi_{+}(x)} & e^{\nu\pi i/2} W_n(x) \\ 
	        0 & e^{-2n\varphi_{-}(x)}\end{pmatrix}  \\
	= \begin{pmatrix} 1 & 0\\ \frac{e^{-\nu\pi i/ 2}}{W_n(x)}e^{-2n\varphi_-(x)} & 1\end{pmatrix}
\begin{pmatrix} 0& e^{\nu\pi i/2}W_n(x)\\ -\frac{e^{-\nu\pi i/2}}{W_n(x)} & 0\end{pmatrix}
\begin{pmatrix} 1 & 0\\ \frac{e^{-\nu\pi i/2}}{W_n(x)}e^{-2n\varphi_{+}(x)} &1\end{pmatrix},
\end{multline*}
while the jump matrix on  $(0,1)$ factorizes as
\begin{multline*}
\begin{pmatrix} e^{-2n\varphi_{+}(x)} & e^{-\nu\pi i/2}W_n(x)  \\ 
	        0 & e^{-2n\varphi_{-}(x)}\end{pmatrix} \\
	= \begin{pmatrix} 1 & 0\\ \frac{e^{\nu\pi i/2}}{W_n(x)} e^{-2n\varphi_{-}(x)}& 1\end{pmatrix}
\begin{pmatrix} 0& e^{-\nu\pi i/2}W_n(x)\\ -\frac{e^{\nu\pi i/2}}{W_n(x)} & 0\end{pmatrix}
\begin{pmatrix} 1 & 0\\ \frac{e^{\nu\pi i/2}}{W_n(x)} e^{-2n\varphi_{+}(x)} &1\end{pmatrix}.
\end{multline*}

In order to open the lens around $(-1,1)$,
 we need the analytic extension of the function $W_n$ from  \eqref{eq:definition-W}
to $\mathbb C \setminus i \mathbb R$, which we also denote by $W_n$,
\begin{equation}\label{analyticW}
	W_n(z)=\begin{cases} 
	\sqrt{2n} K_{\nu}(n\pi z)e^{n\pi z},& \qquad \Re z > 0, \\
  \sqrt{2n} K_{\nu}(-n\pi z)e^{-n\pi z},& \qquad \Re z < 0. 
\end{cases}
\end{equation}
Note that as $n \to \infty$, see  \eqref{asympKv} and \eqref{analyticW},
\begin{equation} \label{Wnestimate} W_n(z) = \begin{cases} z^{-1/2} ( 1 + \mathcal{O}(1/(nz)), & \Re z > 0, \\
	   (-z)^{-1/2}(1 + \mathcal{O}(1/(nz))), & \Re z < 0, 
		\end{cases} \end{equation}
which explains the factor $\sqrt{2n}$ that we introduced in \eqref{eq:definition-W} and \eqref{analyticW}. 
		
\begin{figure}
\centerline{\includegraphics{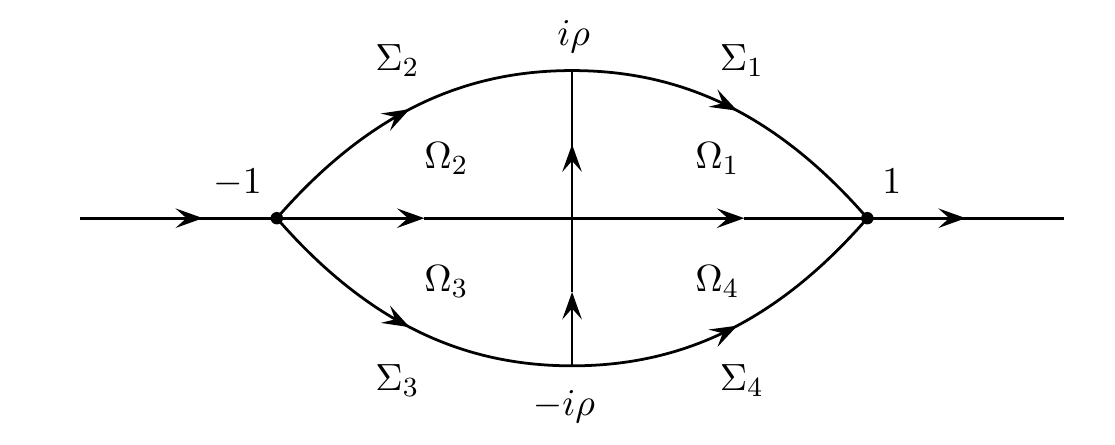}}
\caption{Opening of a lens around $[-1,1]$, and contour $\Sigma_S$ consisting of $\Sigma_1, \ldots, \Sigma_4$, 
the segment $(-i\rho,i\rho)$ and the real line.}
\label{fig_lens2}
\end{figure}

Next, we fix a number $\rho>0$ and we open a lens around $[-1,1]$, which defines contours 
$\Sigma_j$, $j=1, \ldots, 4$
and domains $\Omega_j$, $j=1, \ldots, 4$  as indicated in Figure \ref{fig_lens2}. 

In the fourth transformation we define the matrix valued function $S(z)$:
\begin{align} \label{S}
S(z) = \begin{cases} 
     T(z)\begin{pmatrix} 1 & 0\\ -\frac{e^{\nu\pi i/2}}{W_n(z)} e^{-2n\varphi(z)}& 1\end{pmatrix}, & \text{for } z \in \Omega_1,\\
    T(z) \begin{pmatrix} 1 & 0\\ -\frac{e^{-\nu\pi i/ 2}}{W_n(z)}e^{-2n\varphi(z)} & 1\end{pmatrix}, & \text{for } z \in \Omega_2,\\
     T(z) \begin{pmatrix} 1 & 0\\ \frac{e^{-\nu\pi i/ 2}}{W_n(z)}e^{-2n\varphi(z)} & 1\end{pmatrix}, & \text{for } z \in \Omega_3,\\
     T(z) \begin{pmatrix} 1 & 0\\ \frac{e^{\nu\pi i/2}}{W_n(z)} e^{-2n\varphi(z)}& 1\end{pmatrix}, & \text{for } z \in \Omega_4,\\
     T(z), &  \textrm{elsewhere},
     \end{cases}
\end{align}
using the analytic extension \eqref{analyticW} for the function $W_n(z)$ in each region, 
and $\varphi(z)$ defined in \eqref{phifunction}. 

\begin{remark}
 In order to divide by $W_n(z)$ we need to be careful with possible zeros of this function in the complex plane. 
Following the general theory in \cite[\S 15.7]{Watson}, the Bessel function $K_{\nu}(n\pi z)$ is free from
zeros in the half--plane $|\arg z|\leq \tfrac{\pi}{2}$. Using \eqref{analyticW}, 
we can conclude that $W_n(z)\neq 0$. 
\end{remark}

From the RH problem \ref{RHforT} and \eqref{S} we find that that $S(z)$ 
is the solution of the following RH problem:
\begin{rhp}\label{RHforS}
\begin{itemize}
\item[1)] $S$ is analytic in $\mathbb{C}\setminus \Sigma_S$, where $\Sigma_S$ is depicted in Figure \ref{fig_lens2}.
\item[2)] $S$ satisfies the jump conditions $S_+ = S_- J_S$ where
\begin{align} \label{jumps:Sgeneral}
J_S(z) = 
\begin{cases}
\begin{pmatrix} 1 & 0\\ \frac{e^{\nu\pi i/2}}{W_n(z)} e^{-2n\varphi(z)}& 1\end{pmatrix},
	& \quad z\in\Sigma_1\cup \Sigma_4,\\
\begin{pmatrix} 1 & 0\\ \frac{e^{-\nu\pi i/ 2}}{W_n(z)}e^{-2n\varphi(z)} & 1\end{pmatrix},
	& \quad z\in\Sigma_2\cup\Sigma_3,\\
\begin{pmatrix} 0& e^{\nu\pi i/2}W_n(x)\\ -\frac{e^{-\nu\pi i/2}}{W_n(x)} & 0\end{pmatrix},
	& \quad z \in (-1,0),\\
\begin{pmatrix} 0& e^{-\nu\pi i/2}W_n(x)\\ -\frac{e^{\nu\pi i/2}}{W_n(x)} & 0\end{pmatrix},
  & \quad z \in (0,1),\\
\begin{pmatrix} 1 & e^{\nu\pi i/2}e^{2n \varphi(z)} W_n(z) \\ 0 & 1\end{pmatrix},
  & \quad z\in(-\infty,-1),\\
\begin{pmatrix} 1 & e^{-\nu\pi i/2}e^{2n \varphi(z)} W_n(z) \\ 0 & 1\end{pmatrix},
& \quad z\in(1,\infty),\\
\begin{pmatrix}
             1 & 0\\
	     j_1(z) & 1
             \end{pmatrix},  & \quad z\in(0, i \rho),\\
\begin{pmatrix}
             1 & 0\\
	     j_2(z) & 1
             \end{pmatrix},   & \quad z \in (-i \rho, 0). 
\end{cases}
\end{align}

Here 
\begin{equation}
\label{eq:definition-j1}
j_1(z)=\frac{e^{\nu\pi i/2} e^{-2n\varphi_-(z)}}{W_{n,-}(z)}-\frac{e^{-\nu\pi i/2}e^{-2n\varphi_+(z)}}{W_{n,+}(z)},
	\qquad z \in (0, i \rho),
\end{equation}
and
\begin{equation}
\label{eq:definition-j2}
j_2(z)=-\frac{e^{\nu\pi i/2} e^{-2n\varphi_-(z)}}{W_{n,-}(z)}+\frac{e^{-\nu\pi i/2}e^{-2n\varphi_+(z)}}{W_{n,+}(z)},
	\qquad z \in (-i \rho, 0),
\end{equation}
using the appropriate values of $\varphi_{\pm}(z)$ and $W_{n,\pm}(z)$ in each case. The imaginary axis is
oriented upwards, and so for $z \in i\mathbb R$, we have that $\varphi_+(z)$ and $W_{n,+}(z)$ 
($\varphi_-(z)$ and $W_{n,-}(z)$) denote the limiting value from  the left (right) half-plane.

\item[3)] As $z\rightarrow\infty$,
\begin{equation*}
S(z)=I+\mathcal{O}(1/z).
\end{equation*}
\item[4)] $S(z)$ remains bounded as $z \to 0$ with $\Im z > 0$, and 
\begin{equation} \label{at0:Sgeneral}
S(z)=\begin{pmatrix}
\mathcal{O}(z^{\nu}) & \mathcal{O}(z^{-\nu}) \\
\mathcal{O}(z^{\nu}) & \mathcal{O}(z^{-\nu}) \end{pmatrix}, \quad \text{ as } z \to 0 \text{ with } \Im z < 0.
\end{equation}
\end{itemize}
\end{rhp}

Note that as a consequence of the definition of $\varphi(z)$ in
\eqref{phifunction} and formula \eqref{phig}, $\Im \varphi(x)$ is decreasing on $[-1,1]$. Because 
of the Cauchy--Riemann equations, $\Re \varphi(z)>0$ as we move away from the interval.

We may and do assume that the lens is small enough such that $\Re \varphi(z) > 0$ on the lips
of the lens. Then it follows from \eqref{Wnestimate} and \eqref{jumps:Sgeneral}
that the jump matrix $J_S$ on the lips of the lens tends to $I$ at an exponential rate as $n \to \infty$,
if we stay away from the endpoints $\pm 1$. Also the jump matrix on $(-\infty,-1)$
and $(1, \infty)$ tends to the identity matrix.
Thus for any $\delta > 0$, there is a constant $c > 0$ such that
\begin{equation} \label{JSasymp}
	J_S(z) = I + \mathcal{O}(e^{-cn}), \qquad z \in \Sigma_S \setminus ([-1,1] \cup [-i \rho, i\rho] \cup D(\pm 1, \delta)).
	\end{equation}

The condition \eqref{at0:Sgeneral} needs some explanation, since \eqref{at0:Tgeneral} and \eqref{S} 
at first sight lead to the behavior
$S(z)=\begin{pmatrix}
\mathcal{O}(1) & \mathcal{O}(z^{-\nu}) \\
\mathcal{O}(1) & \mathcal{O}(z^{-\nu}) \end{pmatrix}$ as $z \to 0$  with $\Im z < 0$.
However, a cancellation takes place for the entries in the first column,
as can be checked from the jump conditions  for $S$, see \eqref{jumps:Sgeneral} 
on the intervals $(-1,0)$ and $(0,1)$. Since $S$ remains bounded as $z \to 0$ with $\Im z > 0$,
and 
\[ S_-(z) = S_+(z) \begin{pmatrix} 0 & \mathcal{O}(z^{-\nu}) \\ \mathcal{O}(z^{\nu}) & 0 \end{pmatrix},
		\quad \text{ as } z \to 0, \]
		one finds \eqref{at0:Sgeneral}.

\subsection{Global parametrix}\label{Sec_global}

If we ignore the jump matrices in the RH problem for $S$ except for the one on the interval $[-1,1]$, we arrive at
the following RH problem for a $2 \times 2$ matrix valued function $N$:
\begin{rhp}\label{RHforNgeneral}
\begin{itemize}
\item[1)] $N$ is analytic in $\mathbb{C}\setminus [-1,1]$.
\item[2)] $N$ satisfies the jump conditions
\begin{align*}
N_+(x)&=N_-(x)
\begin{cases}
\begin{pmatrix} 0 & e^{\nu\pi i/2}W_n(x) \\ -\frac{e^{-\nu\pi i/2}}{W_n(x)} & 0 \end{pmatrix},
\quad x\in(-1,0), \\
\begin{pmatrix} 0 & e^{-\nu\pi i/2}W_n(x) \\ -\frac{e^{\nu\pi i/2}}{W_n(x)} & 0 \end{pmatrix},
\quad x\in(0,1).
\end{cases}
\end{align*}
\item[3)] As $z\rightarrow\infty$,
\begin{equation*}
N(z)=I+\mathcal{O}(1/z).
\end{equation*}
\end{itemize}
\end{rhp}

We solve the RH problem for $N$ by means of two Szeg\H{o} functions $D_{1,n}$ and $D_2$, see also
\cite{KMcVV}, that are associated with $W_n$ and $e^{- \sgn(x) \nu \pi i/2}$, respectively.

The first Szeg\H{o} function $D_1 = D_{1,n}$ is defined by 
\begin{equation}
\label{eq:D1(z)}
D_{1,n}(z)=\exp \left(\frac{(z^2-1)^{1/2}}{2\pi}\int_{-1}^{1}\frac{\log W_n(x)}{\sqrt{1-x^2}}\frac{dx}{z-x}\right),
\end{equation}
which is defined and analytic for $z \in \mathbb C \setminus [-1,1]$. It satisfies 
\begin{equation}\label{D1plusminus}
D_{1,n+}(x)D_{1,n-}(x)=W_n(x), \qquad x \in (-1,1). 
\end{equation}
It follows from \eqref{eq:D1(z)} that $D_{1,n}$ has no zeros in $\mathbb C \setminus [-1,1]$ and  
\begin{equation} \label{D1infinity}
   D_{\infty,n} := \lim_{z\to\infty} D_{1,n}(z) =
	\exp \left(\frac{1}{2\pi}\int_{-1}^1 \frac{\log W_n(x)}{\sqrt{1-x^2}} dx\right) \in (0,\infty).
\end{equation} 
In what follows we are not going
to indicate the $n$-dependence in the notation for $D_{1,n}$ and $D_{\infty,n}$, since 
the dependence on $n$ is only mildly. Indeed, because of \eqref{Wnestimate}
we have that $D_{1,n}$ tends to the Szeg\H{o} function for the weight $|x|^{-1/2}$ with
a rate as given in the following lemma.

\begin{lemma} \label{lem:D1nlimit}	We have
\begin{align} \label{D1nlimit} 
	D_{1,n}(z) & = \left( \frac{z + (z^2-1)^{1/2}}{z} \right)^{1/4} \left( 1 + \mathcal{O}\left(\frac{\log n}{n}\right) \right), \\
	  \label{Dinftylimit}
	D_{\infty,n} & = 2^{1/4} + \mathcal{O}\left(\frac{\log n}{n}\right), 
	\end{align}
	as $n \to \infty$, with $\mathcal{O}$-term that is uniform for 
	$z \in \mathbb C \setminus ([-1,1] \cup D(0, \delta)\cup D(\pm 1,\delta))$
	for any $\delta > 0$.
	\end{lemma}
\begin{proof}
The Szeg\H{o} function for $|x|^{-1/2}$ is
\[ D(z; |x|^{-1/2}) = \exp \left(\frac{(z^2-1)^{1/2}}{2\pi} \int_{-1}^1 \frac{ \log |x|^{-1/2}}{\sqrt{1-x^2}} \frac{dx}{z-x} \right)
	= \left( \frac{z + (z^2-1)^{1/2}}{z} \right)^{1/4}. \]
and so
\begin{equation} \label{D1nformula} \left( \frac{z + (z^2-1)^{1/2}}{z} \right)^{-1/4} D_{1,n}(z) = 
	\exp  \left(\frac{(z^2-1)^{1/2}}{2\pi} 
	\int_{-1}^1 \frac{ \log (|x|^{1/2} W_n(x))}{\sqrt{1-x^2}} \frac{dx}{z-x} \right).
\end{equation}
Because of \eqref{Wnestimate} there exist $c_0, c_1 > 0$ 
\[ \left| |x|^{1/2} W_n(x) - 1 \right| \leq \frac{c_1}{n|x|} < \frac{1}{2}, \qquad |x| \geq \frac{c_0}{n}.  \]
Then also for some $c_2 > 0$,
\[ \left| \log(|x|^{1/2} W_n(x))\right| \leq \frac{c_2}{n|x|}, \qquad |x| \geq \frac{c_0}{n}.  \]
It follows that
\begin{align*} 
	\left| \int_{c_0/n}^1 \frac{ \log(|x|^{1/2} W_n(x))}{\sqrt{1-x^2}} \frac{dx}{z-x}  \right|
	& \leq \frac{c_2}{\dist(z, [-1,1]) n} \int_{c_0/n}^1 \frac{1}{x\sqrt{1-x^2}} dx  \\
	& 	\leq \frac{c_3}{\dist(z, [-1,1])} \frac{\log n}{n} \end{align*}
	with a constant $c_3$ that is independent of $n$ and $z$. 
	By deforming the integration path into the complex plane in such a way that it stays at a
	certain distance from $z$, and applying similar estimates we find
\begin{equation} \label{D1nestimate1} 
	\left| \int_{c_0/n}^1 \frac{ \log(|x|^{1/2} W_n(x))}{\sqrt{1-x^2}} \frac{dx}{z-x}  \right|
		\leq \frac{c_4}{|z|} \frac{\log n}{n} \end{equation}	
	with a constant that is independent of 
	$z \in \mathbb C \setminus ([-1,1] \cup D(0, \delta)\cup D(\pm 1,\delta))$.
Similarly
\begin{equation} \label{D1nestimate2} 
	\left| \int_{-1}^{-c_0/n} \frac{ \log (|x|^{1/2} W_n(x))}{\sqrt{1-x^2}} \frac{dx}{z-x}  \right|
		\leq \frac{c_5}{|z|} \frac{\log n}{n}. \end{equation}
		
Near $x=0$ we use \eqref{near0Kv} and \eqref{eq:definition-W} to find a $c_6 > 0$ such
that
\[  c_6 |nx|^{1/2 - \nu}\leq |x|^{1/2} W_n(x) \leq 1, \qquad |x| \leq \frac{c_0}{n}. \]
The upper bound follows from the fact that $0 < K_{\nu}(s) \leq K_{1/2}(s)$ if $0\leq \nu<1/2$ and
$s > 0$ and the explicit formula for $K_{1/2}(s)$ see \cite[10.37.1,10.39.2]{DLMF}.
Then
\[ \left| \log(|x|^{1/2} W_n(x)) \right| \leq 
	\left|\log c_6 + \left(\tfrac{1}{2} - \nu\right) \log |nx|\right|, 
	\qquad |x| \leq \frac{c_0}{n} \]
and
\begin{equation} \label{D1nestimate3} 
	\left| \int_{-c_0/n}^{c_0/n} \frac{ \log |x|^{1/2} W_n(x)}{\sqrt{1-x^2}} \frac{dx}{z-x}  \right|
		 \leq \frac{2}{|z|} 		
		 \int_{-c_0/n}^{c_0/n}  \left| \log c_6 +\left(\tfrac{1}{2} - \nu\right) \log |nx| \right|  dx
		\leq \frac{c_7}{|z|} \frac{1}{n} \end{equation}
for some new constant $c_7 > 0$.

Combining the estimates \eqref{D1nestimate1}, \eqref{D1nestimate2}, and \eqref{D1nestimate3}, 
we get
\[ 	 \left| \frac{(z^2-1)^{1/2}}{2\pi} \int_{-1}^1 \frac{ \log(|x|^{1/2} W_n(x))}{\sqrt{1-x^2}} \frac{dx}{z-x} \right|
	 = \mathcal{O}\left(\frac{ \log n}{n}\right) \]
with a $\mathcal{O}$ term that is uniform for $|z| > \delta$ $|z\pm 1| > \delta$, and so by \eqref{D1nformula}
\[ \left( \frac{z + (z^2-1)^{1/2}}{z} \right)^{-1/4} D_{1,n}(z) = 
   \exp\left( \mathcal{O}\left(\frac{ \log n}{n}\right)\right) = 1 + \mathcal{O}\left(\frac{ \log n}{n}\right) \]
	as claimed in \eqref{D1nlimit}.
	
Since \eqref{D1nlimit} is uniform for $|z| > \delta$, $|z\pm 1| > \delta$, we can let $z \to \infty$, and
obtain \eqref{Dinftylimit}.
\end{proof}

The second Szeg\H{o} function $D_2$ corresponds to the weight $e^{\pm\nu\pi i/2}$, and is defined as 
\begin{equation} \label{eq:D2(z)}
D_2(z)=\left(\frac{\sqrt{z^2-1}-i}{\sqrt{z^2-1}+i}\right)^{\nu/4}, \qquad z \in \mathbb C \setminus [-1,1],
\end{equation}
with the branch of the square root that is positive for real $z > 1$. 
It is not difficult to check that $z \mapsto w = D_2(z)$ is the conformal mapping from $\mathbb C \setminus [-1,1]$ onto
the sector $- \frac{\nu \pi}{4} < \arg w < \frac{\nu\pi}{4}$ that maps $z= 0+$ to $w=0$, $z=0-$ to $w=\infty$,
$z = \pm 1$ to $e^{\mp \frac{\nu \pi}{4}}$ and $z=\infty$ to $w=1$. 

The Szeg\H{o} function $D_2$ is related to the function $\psi$ from \eqref{complexpsi}.
\begin{lemma}
We have
\begin{equation} \label{D2andpsi}
	\log D_2(z) = \begin{cases} - \frac{\nu \pi}{2} \psi(z) - \frac{\nu \pi i}{4}, & \Re z > 0, \, \Im z > 0, \\
		\frac{\nu \pi}{2} \psi(z) - \frac{\nu \pi i}{4}, & \Re z > 0, \, \Im z <0, \\
		- \frac{\nu \pi}{2} \psi(z) +  \frac{\nu \pi i}{4}, & \Re z < 0, \, \Im z > 0, \\
		\frac{\nu \pi}{2} \psi(z) + \frac{\nu \pi i}{4}, & \Re z < 0, \, \Im z < 0.
	\end{cases} \end{equation}
\end{lemma}
\begin{proof}
This follows from \eqref{complexpsi} and \eqref{eq:D2(z)} by straightforward calculation.
\end{proof}

It follows from \eqref{D2andpsi} that $D_2$ satisfies
\begin{equation} \label{D2jump}
 D_{2+}(x) D_{2-}(x) = \begin{cases} e^{\nu\pi i/2}, & \quad x \in (-1,0), \\
	e^{-\nu\pi i/2}, & \quad x \in (0,1), \end{cases} 
	\end{equation}
and, since $\psi(z) \sim \frac{1}{\pi} \log (1/z)$ as $z \to 0$,
\begin{equation} \label{D2at0}
D_{2}(z) = \begin{cases} \mathcal{O}(z^{\nu/2}) & \text{ as } z \to 0 \text{ with } \Im z > 0, \\
	\mathcal{O}(z^{-\nu/2}) & \text{ as } z \to 0 \text{ with } \Im z < 0.
	\end{cases}
\end{equation}

Having $D_1$ and $D_2$ we seek $N$ in the form
\begin{equation} \label{solutionNgeneral}
N(z) = D_{\infty}^{\sigma_3} N_0(z) \left(D_1(z) D_2(z) \right)^{-\sigma_3}.
\end{equation}
Then  $N$ satisfies the RH problem \ref{RHforNgeneral} if and only if $N_0$
satisfies the following standard RH problem:
\begin{rhp}\label{RHforN0general}
\begin{itemize}
\item[1)] $N_0$ is analytic in $\mathbb{C}\setminus [-1,1]$.
\item[2)] $N_0$ satisfies the jump conditions
\[ N_{0+}(x)= N_{0-}(x)\begin{pmatrix} 0 & 1\\ -1 & 0 \end{pmatrix},
\qquad x\in(-1,1).
\]
\item[3)] $N_0(z) = I + \mathcal{O}(1/z)$ as $z\rightarrow\infty$.
\end{itemize}
\end{rhp}

The RH problem for $N_0$ has the explicit solution (see for instance \cite[Section 7.3]{Deift}): 
\begin{equation} \label{eq:N00}
N_0(z)=\begin{pmatrix}
\frac{\beta(z)+\beta(z)^{-1}}{2} & \frac{\beta(z)-\beta(z)^{-1}}{2i}\\
-\frac{\beta(z)-\beta(z)^{-1}}{2i} & \frac{\beta(z)+\beta(z)^{-1}}{2}
\end{pmatrix}, \quad \text{ with }
\beta(z)=\left(\frac{z-1}{z+1}\right)^{1/4}, 
\end{equation}
for $z \in \mathbb C \setminus [-1,1]$,
and we take the branch of the fourth root that is analytic in $\mathbb C \setminus [-1,1]$
and that is real and positive for $z > 1$. 
Note that we can also write
\begin{equation} \label{eq:N0}
N_0(z)= \frac{1}{\sqrt{2}(z^2-1)^{1/4}} \begin{pmatrix} f(z)^{1/2} & i f(z)^{-1/2} \\ -i f(z)^{-1/2} & f(z)^{1/2} \end{pmatrix}
\end{equation}
where 
\begin{equation}\label{fz}
f(z) = z + (z^2-1)^{1/2} 
\end{equation}
is the conformal map from $\mathbb C \setminus [-1,1]$ to the exterior of the unit disk.

\subsection{Fifth transformation}

Around the endpoints $z=\pm 1$ we build Airy parametrices $P_{\Ai}$ in the usual way.
We take $\delta > 0$ sufficiently small, and $P_{\Ai}$ is defined and analytic in $D(\pm 1, \delta) \setminus \Sigma_S$
such that it has the same jumps as $S$ on $\Sigma_S \cap D(\pm 1, \delta)$, 
and such that 
\begin{equation} \label{matching}
	P_{\Ai}(z) = N(z) (1 + \mathcal{O}(n^{-1})),  \quad \text{uniformly for } |z \pm 1| = \delta,
	\end{equation}
	as $n \to \infty$. We refer 
the reader for instance to the monograph by Deift \cite[\S 7.6]{Deift} for details.

In the fifth transformation we put
\begin{equation} \label{Q}
	Q = \begin{cases} SN^{-1}, & \text{ outside the disks $D(\pm 1, \delta)$,} \\
	S P_{\Ai}^{-1}, & \text{ inside the disks.}
	\end{cases} \end{equation}
Then $Q$ is defined and analytic outside of a contour consisting of $\Sigma_S$ and
two circles around $\pm 1$. The construction of the Airy parametrix is such that
it has the same jump as $S$ inside the circles. As a result $Q$ is analytic inside
the two disks. Also  $S$ and $N$ have the same jump on $(-1,1)$ and it follows that
$Q$ is analytic across $(-1,1)$. Therefore $Q$ is analytic in $\mathbb C \setminus \Sigma_Q$ where 
$\Sigma_Q$ consists of two circles around $\pm 1$, the parts of $(-\infty, -1)$, $\Sigma_j$, 
$j=1,\ldots, 4$ and $(1, \infty)$ outside of these circles, 
and the segment $(-i \rho, i \rho)$ on the imaginary axis. See Figure \ref{figQ}.

\begin{figure}
\centerline{\includegraphics{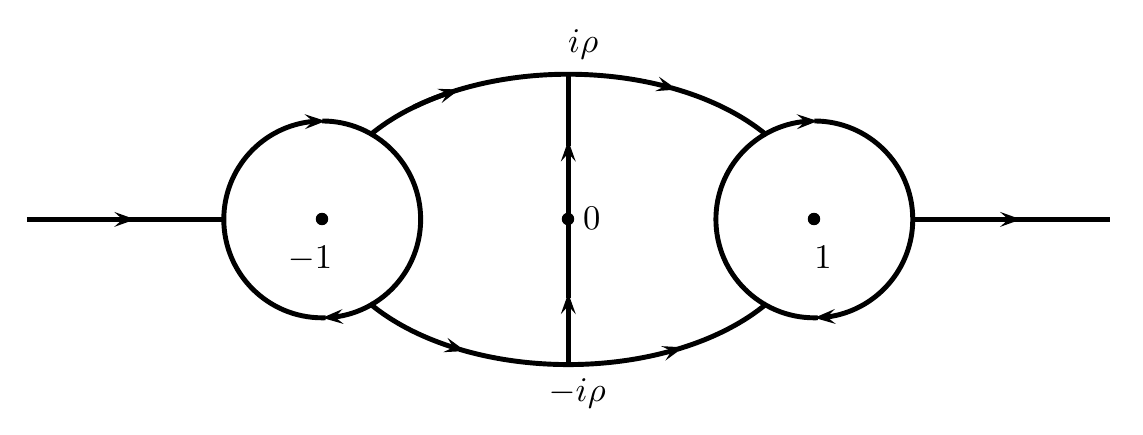}}
\caption{Contour $\Sigma_Q$}
\label{figQ}
\end{figure}

From the RH problem \ref{RHforS} for $S$ and \eqref{Q} it then follows that
$Q$ solves the following RH problem.

\begin{rhp}\label{RHforQ}
\begin{itemize}
\item[1)] $Q : \mathbb C \setminus \Sigma_Q \to \mathbb C^{2 \times 2}$ is analytic.
\item[2)] $Q$ satisfies the jump condition
$Q_+ = Q_- J_Q$ on $\Sigma_Q$ where
\begin{align*}
J_Q(z) = 
\begin{cases}
N(z) P_{\Ai}^{-1}(z), & \text{ for $z$ on the circles}, \\
N(z) \begin{pmatrix} 1 & 0\\ j_1(z) & 1 \end{pmatrix} N^{-1}(z) & \text{ for } z \in (0, i \rho),\\
N(z) \begin{pmatrix} 1 & 0\\ j_2(z) & 1 \end{pmatrix} N^{-1}(z) & \text{ for } z \in (-i \rho, 0), \\
N(z) J_S(z) N(z)^{-1}, & \text{ elsewhere on $\Sigma_Q$.}
\end{cases}
\end{align*}
Here $j_1$ and $j_2$ are given by \eqref{eq:definition-j1} and \eqref{eq:definition-j2}.
\item[3)] As $z\rightarrow\infty$,
\begin{equation*}
Q(z)=I+\mathcal{O}(1/z).
\end{equation*}
\item[4)] $Q(z) = \mathcal{O}(1)$ as $z \to 0$. 
\end{itemize}
\end{rhp}
In the behavior around $0$ there is no longer a distinction between the upper
and lower half planes, and $Q$ remains bounded in all directions.

We note that
\begin{equation} \label{JQasymp1}
	J_Q(z) = I + \mathcal{O}(n^{-1}), \qquad \text{ for $z$ on the circles} 
\end{equation}
because of the matching property \eqref{matching}. We also note that
\begin{equation} \label{JQasymp2}
	J_Q(z) = I + \mathcal{O}(e^{-cn}), \qquad \text{ on } \Sigma_Q \setminus (\partial D(\pm 1, \delta) \cup  [-i \rho, i\rho] )
	\end{equation}
because of \eqref{JSasymp}, \eqref{solutionNgeneral}, and Lemma \ref{lem:D1nlimit}.

The jump matrix $J_Q$ on the imaginary axis can be rewritten as (we use \eqref{solutionNgeneral}):
\begin{align} \label{jumpQ1} 
	J_Q(z) = D_{\infty}^{\sigma_3}
	N_0(z) \begin{pmatrix} 1 & 0 \\ j_{1,2}(z) (D_1(z) D_2(z))^2 & 1 \end{pmatrix} 
		N_0^{-1}(z) D_{\infty}^{-\sigma_3}, \qquad z \in (-i \rho, i \rho),
		\end{align}
		with $j_1$ on $(0, i\rho)$, and $j_2$ on $(-i \rho,0)$.

The entry $j_{1,2}(z) (D_1(z) D_2(z))^2$ in \eqref{jumpQ1} depends on $n$,
and tends to $0$ as $n \to\infty$ for every $z \in (-i \rho, 0) \cup (0, i \rho)$,
but not in a uniform way. Hence, further analysis is needed in the next section. 
A similar situation is studied in \cite[Section 5]{BB}, where the jump 
on the imaginary axis has the same structure and approaches the identity matrix
at a rate $1/\log(n)$ as $n\to \infty$. In that case no local parametrix near the origin is needed.

\subsection{Local parametrix near $z=0$}
The construction of a local parametrix in a neighborhood of the origin 
follows the idea exposed in \cite{KMcL}. We take $\varepsilon > 0$ with 
\[ \varepsilon < \min \left( \tfrac{1}{2 e}, \tfrac{\rho}{3} \right) \]
and we build a local parametrix $P$ defined in a neighborhood $|z| < 3\varepsilon$ of $0$.
We use a cut-off function $\chi(z)$ on $i \mathbb R$ such that
\begin{enumerate}
\item[(a)] $\chi : i \mathbb R \to \mathbb R$ is a $C^{\infty}$ function,
\item[(b)] $0 \leq \chi(z) \leq 1$ for all $z \in i \mathbb R$,
\item[(c)] $\chi(z) \equiv 1$ for $z \in (-i \varepsilon, i \varepsilon)$,
\item[(d)] $\chi(z) \equiv 0$ for $z \in \left(-i \infty, -2i\varepsilon\right) \cup \left(2i\varepsilon, i \infty\right)$.
\end{enumerate} 
Then we modify $J_Q$ by multiplying the off-diagonal entry in the middle factor of \eqref{jumpQ1} by $\chi(z)$,
and in addition we use this as a jump matrix in the full imaginary axis. Thus
\begin{equation} \label{jump:P} 
	J_{P}(z) = D_{\infty}^{\sigma_3} 
	N_0(z) \begin{pmatrix} 1 & 0 \\ j_{1,2}(z) (D_1(z) D_2(z))^2 \chi(z) & 1 \end{pmatrix} N_0^{-1}(z) D_{\infty}^{-\sigma_3},
		\qquad z \in i \mathbb R,
\end{equation}
with $j_1$ on $i \mathbb R^+$ and $j_1$ on $i \mathbb R^-$.

Then the RH problem for the local parametrix $P$ at the origin is:
\begin{rhp} \label{RHforP}
\begin{itemize}
\item[\rm 1)] $P : \{ z\in \mathbb C \mid -1 < \Re z < 1 \} \setminus i \mathbb R \to \mathbb C^{2 \times 2}$ is analytic.
\item[\rm 2)] $P$ satisfies the jump condition
\begin{equation} \label{jumpP}
 P_+(z)= P_-(z) J_P(z), \quad z\in i \mathbb R,
\end{equation}
where $J_P(z)$ is given by \eqref{jump:P}.
\item[\rm 3)] $P(z) = I + \mathcal{O} \left( \epsilon_n \right)$ as $n \to \infty$ 
uniformly for $|z| = 3 \varepsilon$ with $\epsilon_n$ given by \eqref{epsilonn}. 
\end{itemize}
\end{rhp}

\begin{proposition} \label{propo8}
The RH problem \ref{RHforP} has a solution for $n$ large enough.
\end{proposition}
The rest of this subsection is devoted  to the proof of Proposition \ref{propo8}.
It takes a number of steps and it is the most technical part of the paper. 

\subsubsection{RH problem for $\widehat P$}

We introduce a matrix $\widehat{P}(z)$ in the following way:
\begin{equation} \label{Phat} 
	P(z) = \begin{cases}
	D_{\infty}^{\sigma_3} N_0(z)  \widehat P(z)  N_0(z)^{-1} D_{\infty}^{-\sigma_3}, &
		\text{for } \Im z < 0, \\[5pt]
	D_{\infty}^{\sigma_3} N_0(z) \begin{pmatrix} 0 & -1 \\ 1 & 0 \end{pmatrix} 
	 \widehat P(z) 
		\begin{pmatrix} 0 & 1 \\ -1 & 0 \end{pmatrix} N_0(z)^{-1} D_{\infty}^{-\sigma_3}, &
		\text{for } \Im z > 0.
		\end{cases}
\end{equation}

The extra factors in \eqref{Phat} for $\Im z > 0$ are introduced in order to compensate the 
jumps of $N_0$ on $[-1,1]$. Then $P$ satisfies the jump condition \eqref{jumpP} in the RH problem \ref{RHforP}
if and only if $\widehat P_+ = \widehat P_- J_{\widehat P}$, where the jump is
\begin{equation} \label{jump:Phat}
J_{\widehat P}(z) = 
	\begin{cases} \begin{pmatrix} 1 & - j_1(z) (D_1(z) D_2(z))^2 \chi(z) \\ 0 & 1 \end{pmatrix}, & \text{ for } z \in i \mathbb R^+, \\[5mm] 
	   \begin{pmatrix} 1 & 0 \\  j_2(z) (D_1(z) D_2(z))^2 \chi(z) & 1 \end{pmatrix}, & \text{ for } z \in i \mathbb R^-.
	\end{cases} \end{equation}
Note the difference in the triangularity structure. So, we look for $\widehat P$ that solves the following RH problem:
\begin{rhp} \label{RHforPhat}
\begin{itemize}
\item[1)] $\widehat P : \mathbb C \setminus i \mathbb R \to \mathbb C^{2 \times 2}$ is analytic.
\item[2)] $\widehat P$ satisfies the jump conditions
\begin{equation} \label{jumpcondition:Phat} 
	\widehat P_+(z)= \widehat P_-(z) J_{\widehat P}(z), \quad z\in i \mathbb R,
\end{equation} 
where $J_{\widehat P}(z)$ is given by \eqref{jump:Phat}.
\item[3)] $\widehat P(z) = I + \mathcal{O}(1/z)$ as $z  \to \infty$.
\end{itemize}
\end{rhp}

Our aim is to show that the RH problem for $\widehat P$ has a solution for $n$ sufficiently large, and
that this solution satisfies in addition
\begin{itemize}
\item[4)] $\widehat P(z) = I + \mathcal{O} \left( \epsilon_n \right)$ as $n \to \infty$, uniformly for $|z| = 3 \varepsilon$.
\end{itemize}
Having $\widehat P$ we define $P$  by \eqref{Phat} in terms of $\widehat P$, and it will satisfy 
the requirements of the RH problem \ref{RHforP}. 

We prove the following result:
\begin{lemma}\label{lem:Phat}
If $0<\nu \leq 1/2$, then for $n$ large enough there exists $\widehat{P}(z)$ that solves the 
RH problem \ref{RHforPhat}, 
and as $n\to\infty$,
 \begin{equation}
|\widehat{P}_{11}(z) - 1| = \mathcal{O} \left( n^{-1/2} (\log n)^{-2\nu-1/2}\right), \quad
 |\widehat{P}_{21}(z) | = \mathcal{O} \left( n^{\nu-1/2} (\log n)^{-\nu-1/2} \right), \nonumber
\end{equation}
in $\mathbb{C}\setminus[-2i\varepsilon,0]$, and 
 \begin{equation}
 |\widehat{P}_{12}(z)| = \mathcal{O} \left( n^{-\nu-1/2}(\log n)^{-\nu-1/2} \right), \qquad
|\widehat{P}_{22}(z) - 1| = \mathcal{O} \left( n^{-1/2}  (\log n)^{-2\nu-1/2}\right),  \nonumber
 \end{equation}
in $\mathbb{C}\setminus[0,2i\varepsilon]$.
\end{lemma}

\begin{remark}
It follows from Lemma \ref{lem:Phat} that $\widehat P(z) = I + \mathcal{O} \left( \epsilon_n \right)$ 
as $n \to \infty$, uniformly for $|z| = 3 \varepsilon$, and because of \eqref{Phat}, the same holds for $P(z)$.
\end{remark}

In the proof of this lemma we will need the following steps:
\begin{enumerate}
 \item We write the jump conditions for $\widehat{P}(z)$ componentwise, and in terms of two integral 
operators $K_1$ and $K_2$.
 \item We estimate the operator norms $\|K_1\|$ and $\|K_2\|$ as $n\to\infty$. This requires estimates for 
the functions $j_1(z)$, $j_2(z)$, $D_1(z)$ and $D_2(z)$, which are uniform as $n\to\infty$ for $y$ in a 
fixed interval around the origin on the imaginary axis. 
 \item We show that the operators $I-K_2K_1$ and $I-K_1K_2$ are invertible for $n$ large enough, and this 
gives the existence and asymptotics of $\widehat{P}$. 
\end{enumerate}

Finally, the estimates for $\widehat{P}(z)$ are used to prove that the matrix $R(z)$, which will
be defined in Section \ref{finaltrans} and which solves the 
Riemann--Hilbert problem \ref{RHforR}, is close to the identity matrix as $n\to\infty$. 

\subsubsection{Integral operators} 
Let us write
\begin{equation}
	\begin{aligned} \label{eta12z}
	\eta_1(z) & = - j_1(z) (D_1(z) D_2(z))^2 \chi(z), & z \in i \mathbb R^+,\\
	\eta_2(z) & = j_2(z) (D_1(z) D_2(z))^2 \chi(z), & z \in i \mathbb R^-.  
	\end{aligned}
\end{equation} 
These functions depend on $n$, since $j_1$, $j_2$ and $D_1$ depend on $n$. Note, however, that $D_2$
and $\chi$ do not depend on $n$.

The jump condition \eqref{jump:Phat}-\eqref{jumpcondition:Phat}  yields that for $j=1,2$,
\begin{equation}\label{jumps_entries_Phat}
 \begin{aligned}
 \widehat{P}_{j1+}(z) &= \begin{cases} \widehat{P}_{j1-}(z), & \text{ for } z \in i \mathbb R^+, \\
\widehat{P}_{j1-}(z) +  \eta_2(z) \widehat{P}_{j2-}(z), & \text{ for } z \in i \mathbb R^-, 
\end{cases}\\
\widehat{P}_{j2+}(z) &= \begin{cases} \widehat{P}_{j2-}(z) + \eta_1(z) \widehat{P}_{j1-}(z), & 
\text{ for } z \in i \mathbb R^+, \\
\widehat{P}_{j2-}(z), & \text{ for } z \in i \mathbb R^-.
\end{cases} 
  \end{aligned}
\end{equation}

Since $\chi(z) = 0$ for $|z| \geq 2 \varepsilon$, we find that $\widehat{P}_{j1}$ is analytic 
in $\mathbb C \setminus [-2i \varepsilon, 0] $, and $\widehat{P}_{j2}$ is analytic in $\mathbb C \setminus [0,2i\varepsilon]$. 
Then by the Sokhotski-Plemelj formula and the asymptotic condition $\widehat P(z) \to I$ as $z \to \infty$, we get
\begin{equation}\label{hatP11P12}
\begin{aligned}
\widehat{P}_{11}(z) & = 1 + \frac{1}{2\pi i} \int_{-2i \varepsilon}^0 \frac{\eta_2(s) \widehat{P}_{12}(s)}{s-z} ds, &
\widehat{P}_{12}(z) & = \frac{1}{2\pi i} \int_0^{2i \varepsilon} \frac{\eta_1(s) \widehat{P}_{11}(s)}{s-z} ds. \\
\widehat{P}_{21}(z) & = \frac{1}{2\pi i} \int_{-2i \varepsilon}^0 \frac{\eta_2(s) \widehat{P}_{22}(s)}{s-z} ds, &
\widehat{P}_{22}(z) & = 1 + \frac{1}{2\pi i} \int_0^{2i \varepsilon} \frac{\eta_1(s) \widehat{P}_{21}(s)}{s-z} ds. 
\end{aligned} 
\end{equation}

We can write the equations in operator form if we introduce two operators
\[ K_1 : L^2([0,2i \varepsilon]) \to L^2([-2i \varepsilon,0]) \qquad \text{ and } \qquad
 K_2 : L^2([-2i \varepsilon,0]) \to L^2([0,2i \varepsilon]) \]
by
\begin{align} \label{K12} 
	(K_1 f)(z) & = \frac{1}{2\pi i}  \int_0^{2i \varepsilon} \frac{\eta_1(s) f(s)}{s-z} ds,
		\qquad f \in L^2([0,2i \varepsilon]), \\ 
	(K_2 g)(z) & = \frac{1}{2\pi i} \int_{-2i \varepsilon}^0 \frac{\eta_2(s) g(s)}{s-z} ds,
		\qquad g \in L^2([-2i \varepsilon,0]).
\end{align}
Then $f_1 = \widehat{P}_{11}$, $g_1 = \widehat{P}_{12}$ should solve
\begin{equation} 
\label{eq:formulas-f1-g1}
	f_1 = 1 + K_2 g_1, \quad g_1 = K_1 f_1 
\end{equation}
and $f_2 = \widehat{P}_{21}$, $g_2 = \widehat{P}_{22}$ should solve
\begin{equation}
\label{eq:formulas-f2-g2}
	f_2= K_2 g_2, \quad g_2= 1 + K_1 f_2. 
\end{equation}

Both $K_1$ and $K_2$ are integral operators between Hilbert spaces with operator norms
\begin{align*} 
	\| K_1 \|^2 =  \int_{-2i \varepsilon}^0 \int_0^{2i \varepsilon} \frac{|\eta_1(s)|^2}{|s-t|^2}  |ds| |dt|, \\
	\| K_2 \|^2 = \int_0^{2i \varepsilon} \int_{-2i \varepsilon}^0 \frac{|\eta_2(s)|^2}{|s-t|^2}  |ds| |dt|.
\end{align*}
The $t$-integrals can be done explicitly. This leads to the estimates (we also change to 
a real integration  variable by putting $s = \pm iy$)
\begin{equation} 
\label{eq:norms-K1-K2}
	\| K_1 \| \leq \left( \int_0^{2\varepsilon}  \frac{|\eta_1(iy)|^2}{y} dy \right)^{1/2}, \quad 
	\| K_2 \| \leq  \left( \int_0^{2\varepsilon} \frac{|\eta_2(-iy)|^2}{y} dy \right)^{1/2}. 
\end{equation}

The next step is to show that both integrals are finite (so that $K_1$ and $K_2$ are 
well-defined bounded operators) and that $\| K_1 K_2 \|$ and $\| K_2 K_1 \|$ tend to $0$ as $n \to \infty$. 
To this end, we need to control the functions $\eta_1$ and $\eta_2$, defined in \eqref{eta12z}.

\subsubsection{The functions $\eta_1(z)$ and $\eta_2(z)$}

The functions $\eta_1$ and $\eta_2$ are defined in terms of $j_1$, $j_2$, $D_1$ and $D_2$, see \eqref{eta12z}. 
In this section we obtain estimates for all these functions for large $n$.

First we write the functions $j_1(z)$ and $j_2(z)$ in terms of Bessel functions. Because of the property 
$K_{\nu}(\overline{z})=\overline{K_{\nu}(z)}$ for real $\nu$, see \cite[\S 10.34.7]{DLMF}, if 
we consider the positive imaginary axis and we write $z=iy$, with $y>0$, then the function $W_n$ 
(recall \eqref{analyticW}) can be written as 
\begin{equation}
\label{eq:Wpm-imaginary}
W_{n,\pm}(iy)= \sqrt{2n} K_{\nu}(\mp n\pi iy)e^{\mp n\pi iy},
\end{equation}
so $W_{n,+}(iy)=\overline{W_{n,-}(iy)}$. Similarly, on the negative imaginary axis,
\begin{equation}
\label{eq:Wpm-imaginary2}
 W_{n,\pm}(-iy)= \sqrt{2n} K_{\nu}(\pm n\pi iy)e^{\mp n\pi iy},
\end{equation}
so again $W_{n,+}(-iy)=\overline{W_{n,-}(-iy)}$. Additionally, we have
\begin{equation}\label{WHankel}
\begin{aligned}
|W_{n,-}(iy)|^2 &= 2n |K_{\nu}(n\pi i y)|^2=\frac{n \pi^2}{2}|H^{(2)}_{\nu}(n\pi y)|^2=
\frac{n \pi^2}{2}\left[J_{\nu}(n\pi y)^2+Y_{\nu}(n\pi y)^2\right],\\
|W_{n,-}(-iy)|^2 &= 2n |K_{\nu}(-n\pi i y)|^2=\frac{n\pi^2}{2}|H^{(1)}_{\nu}(n\pi y)|^2=
\frac{n \pi^2}{2}\left[J_{\nu}(n\pi y)^2+Y_{\nu}(n\pi y)^2\right],
\end{aligned}
\end{equation}
in terms of Hankel functions, see \cite[\S 10.27.8]{DLMF}. We have the following auxiliary result:
\begin{lemma} \label{lem:j1_j2_Bessel}
For $y>0$, the functions $j_1(iy)$ and $j_2(-iy)$ can be written as follows:
\begin{equation*}
\begin{aligned}
 |j_1(iy)|&=\frac{2e^{-2n\Re \varphi_-(iy)}}{\sqrt{2n} \pi}
	\frac{|J_{\nu}(n\pi y)\cos\nu\pi-Y_{\nu}(n\pi y)\sin \nu\pi|}{J^2_{\nu}(n\pi y)+Y^2_{\nu}(n\pi y)},\\
 |j_2(-iy)|&=\frac{2e^{-2n\Re \varphi_-(-iy)}}{\sqrt{2n} \pi}\frac{|J_{\nu}(n\pi y)|}{J^2_{\nu}(n\pi y)+Y^2_{\nu}(n\pi y)}.
\end{aligned}
\end{equation*}
\end{lemma}

\begin{proof}
It follows from \eqref{eq:definition-j1} that $j_1$ can be written as
\begin{equation*}
	j_1(iy) =\frac{e^{-2n\varphi_-(iy)-n\pi iy}}{W_{n,-}(iy)W_{n,+}(iy)}
	\left[e^{\frac{\nu\pi i}{2}+n\pi iy}W_{n,+}(iy) -e^{-\frac{\nu\pi i}{2}-n\pi iy}W_{n,-}(iy)\right],
\end{equation*}
and because of $\varphi_-(z)=\varphi_+(z)-\pi z$ on the imaginary axis, and the fact that 
$W_{n,+}(iy)=\overline{W_{n,-}(iy)}$, the two terms on the right hand side are complex conjugates, so
\begin{equation}\label{j1Im}
 j_1(iy)=\frac{-2i e^{-2n\varphi_-(iy)-n\pi iy}}{|W_{n,-}(iy)|^2} 
	\Im \left[e^{-\frac{\nu\pi i}{2}-n\pi iy}W_{n,-}(iy)\right].
\end{equation}

Using the formula
\begin{equation*}
 K_{\nu}(z)=-\frac{\pi i}{2}e^{-\frac{\nu\pi i}{2}} H_{\nu}^{(2)}(ze^{-\frac{\pi i}{2}}), \qquad 
-\frac{\pi}{2}< \arg  z\leq \pi,
\end{equation*}
in terms of Hankel functions, see \cite[\S 10.27.8]{DLMF} and \eqref{eq:Wpm-imaginary} we observe that
\begin{equation*}
 e^{-\frac{\nu\pi i}{2}-n\pi iy}W_{n,-}(iy)=e^{-\frac{\nu\pi i}{2}} \sqrt{2n} K_{\nu}(n\pi iy)
=-\frac{\sqrt{2n} \pi i\, e^{-\nu\pi i}}{2}\left(J_{\nu}(n\pi y)-iY_{\nu}(n\pi y)\right).
\end{equation*}
Hence, on the positive imaginary axis,
\begin{equation*}
 \Im \left[e^{-\frac{\nu\pi i}{2}-n\pi iy}W_{n,-}(iy)\right]
=-\frac{\sqrt{2n}\pi}{2}(J_{\nu}(n\pi y)\cos\nu\pi-Y_{\nu}(n\pi y)\sin \nu\pi).
\end{equation*}

Using \eqref{j1Im} and \eqref{WHankel}, this proves the first formula. Similarly, for $y>0$,
\begin{equation}\label{j2Im}
 j_2(-iy)=
 \frac{2i e^{-2n\varphi_-(-iy)-n\pi iy}}{|W_{n,-}(-iy)|^2} \Im \left[e^{-\frac{\nu\pi i}{2}-n\pi iy}W_{n,-}(-iy)\right].
\end{equation}

In this case, we use 
\begin{equation*}
 K_{\nu}(z)=\frac{\pi i}{2}e^{\frac{\nu\pi i}{2}} H_{\nu}^{(1)}(ze^{\frac{\pi i}{2}}), \qquad 
	-\pi<\arg  z\leq \frac{\pi}{2},
\end{equation*}
see \cite[10.27.8]{DLMF}, and \eqref{eq:Wpm-imaginary2} to obtain
\begin{equation*}
 e^{-\frac{\nu\pi i}{2}-n\pi iy}W_{n,-}(-iy)=e^{-\nu\pi i/2} \sqrt{2n} K_{\nu}(-n\pi iy)
	=\frac{\sqrt{2n} \pi i}{2}\left(J_{\nu}(n\pi y)+iY_{\nu}(n\pi y)\right),
\end{equation*}
so
\begin{equation*}
 \Im \left[e^{\frac{-\nu\pi i}{2}-n\pi iy}W_{n,-}(-iy)\right]
=\frac{\sqrt{2n} \pi}{2}J_{\nu}(n\pi y).
\end{equation*}

We use \eqref{j2Im} and \eqref{WHankel}, and this completes the proof. 
\end{proof}
Next, we will obtain estimates of the previous functions $j_1$ and $j_2$ for large $n$.

\begin{lemma}\label{lem:asymptotic_j1_j2}
 For $0<\nu\leq 1/2$ there exist constants $C_\nu, C'_\nu>0$ such that for all $s > 0$ we have
\begin{equation*}
\begin{aligned}
 \frac{|J_{\nu}(s)\cos\nu\pi-Y_{\nu}(s)\sin \nu\pi|}{J_{\nu}(s)^2+Y_{\nu}(s)^2}
 &\leq C_{\nu}\, \frac{s^{\nu}(1+s^{1-2\nu})}{1+s^{1/2-\nu}},\\
 \frac{|J_{\nu}(s)|}{J_{\nu}(s)^2+Y_{\nu}(s)^2}&\leq C'_{\nu}\, \frac{s^{3\nu}(1+s^{1-2\nu})}{1+s^{1/2+\nu}}.
 \end{aligned}
\end{equation*}
 \end{lemma}
\begin{proof}
For the proof, we consider the following expansions: as $s\to 0^+$,
\begin{equation}\label{asympJ0}
 J_{\nu}(s)=\frac{s^{\nu}}{2^{\nu}\Gamma(\nu+1)}\left(1+\mathcal{O}\left(s^{-1}\right)\right),  \quad \nu\neq -1,-2,\ldots
\end{equation}
and for $\nu<1$ we have
\begin{equation}\label{asympY0}
 Y_{\nu}(s)=-\frac{\Gamma(\nu)}{\pi}\left(\frac{s}{2}\right)^{-\nu}
+  \mathcal{O}(s^{\nu}).
\end{equation}

As $s\to\infty$, we have
\begin{equation}\label{asympJYinf}
J_{\nu}(s)= \left(\frac{2}{\pi s}\right)^{1/2}\cos\omega \, \left(1+\mathcal{O}\left(s^{-1}\right)\right), \qquad 
Y_{\nu}(s)= \left(\frac{2}{\pi s}\right)^{1/2}\sin\omega \, \left(1+\mathcal{O}\left(s^{-1}\right)\right),
\end{equation}
where $\omega=s-\frac{\nu\pi}{2}-\frac{\pi}{4}$.
See for instance \cite[formulas 10.7.3--4, 10.17.3--4]{DLMF}.

From this, it follows that
\begin{equation}\label{asymp:Mnu}
\begin{aligned}
J_{\nu}(s)^2+Y_{\nu}(s)^2 & = \frac{\Gamma(\nu)^2}{\pi^2}\left(\frac{s}{2}\right)^{-2\nu} + \mathcal{O}(1),  &  s\to 0,\\
J_{\nu}(s)^2+Y_{\nu}(s)^2 & =  \frac{2}{\pi s}+\mathcal{O}\left(s^{-2}\right), & s\to \infty.
\end{aligned}
\end{equation}

From \eqref{asymp:Mnu}, we claim that there exist two constants $C_{1,\nu},C_{2,\nu}>0$ such that
\begin{equation*}
 C_{1,\nu}\,\frac{s^{-2\nu}}{1+s^{1-2\nu}}\leq J_{\nu}(s)^2+Y_{\nu}(s)^2 \leq C_{2,\nu}\,\frac{s^{-2\nu}}{1+s^{1-2\nu}}, \qquad s>0.
\end{equation*}

Using a similar argument, we have
\begin{equation*}
|J_{\nu}(s)|\leq C_{3,\nu}\,\frac{s^{\nu}}{1+s^{1/2-\nu}},
\end{equation*}
and also
\begin{equation*}
|J_{\nu}(s)\cos\nu\pi-Y_{\nu}(s)\sin \nu\pi|\leq C_{4,\nu}\,\frac{s^{-\nu}}{1+s^{1/2-\nu}},
\end{equation*}
and putting all the estimates together we get the bounds in the lemma.
\end{proof}

As a consequence of Lemma \ref{lem:j1_j2_Bessel} and Lemma \ref{lem:asymptotic_j1_j2} 
we obtain the following bounds for $j_1$ and $j_2$ for $y>0$:
\begin{equation}
\label{eq:formulas-j1-j2}
\begin{aligned}
 |j_1(iy)|&\leq C_{\nu}\, \frac{2e^{-2n\Re \varphi_-(iy)}}{\sqrt{2n} \pi}   
		\frac{(n\pi y)^{\nu}(1+(n\pi y)^{1-2\nu})}{1+(n\pi y)^{1/2-\nu}},\\
 |j_2(-iy)|&\leq C'_{\nu}\, \frac{2e^{-2n\Re \varphi_-(-iy)}}{\sqrt{2n} \pi} 
			\frac{(n\pi y)^{3\nu}(1+(n\pi y)^{1-2\nu})}{1+(n\pi y)^{1/2+\nu}}.
\end{aligned}
\end{equation}

Next, we need an estimate for $D_1(z)$ (see formula \eqref{eq:D1(z)}),
with $z = iy$, $y \in [-\rho, \rho]$ where we recall that $\pm i \rho$ is the intersection
of the lens with the imaginary axis.

\begin{lemma}
\label{lem:lemmaD1}
For $0<\nu\leq 1/2$, there exists a 
constant $C_\nu$ such that for all sufficiently large $n$,
\begin{equation}\label{boundD1}
|D_1(iy)|^2\leq C_\nu\,\frac{ n^{1/2-\nu} |y|^{-\nu}}{1+(n|y|)^{1/2-\nu}}, \qquad y \in [-\rho, \rho].
 \end{equation}  
\end{lemma}

\begin{proof}
We write first $z=iy$ with $y>0$ in \eqref{eq:D1(z)} and use the parity of the 
function $W_n$ to get the following expression:
\begin{equation}\label{D1y}
 D_1(iy)=\exp\left(\frac{y(y^2+1)^{1/2}}{2\pi}\int_0^1 \frac{\log W_n(x)}{\sqrt{1-x^2}}\frac{dx}{x^2+y^2}\right).
\end{equation}
Using the asymptotic expansions \eqref{asympJ0}, \eqref{asympY0} and \eqref{asympJYinf}, we claim that 
there exist two constants $C_1$ and $C_2$, depending on $\nu$, such that $W_n(x)$ satisfies
\begin{equation*}
W_n(x)\leq C_1 |x|^{-1/2}, \qquad |n \pi x|\geq 1,
\end{equation*}
and
\begin{equation*}
W_n(x)\leq C_2 n^{1/2 - \nu} |x|^{-\nu}, \qquad |n \pi x|\leq 1.
\end{equation*}

Since $\nu\leq 1/2$, both bounds hold uniformly for $n\pi x>0$. 
Since the integrand in \eqref{D1y} is a real function, we can bound $D_1(iy)$ from above by another Szeg\H{o} function:
$$
D_1(iy)^2\leq D(iy;C_1 |\pi x|^{-1/2})^2=C_1 \pi^{-1/2} D(iy;|x|^{-1/2})^2.
$$

This last Szeg\H{o} function is explicit, since for a general exponent $\alpha>-1$ we have
\begin{equation}
\label{eq:SzegoD}
D(z;|x|^{\alpha})=\left(\frac{z}{z+\sqrt{z^2-1}}\right)^{\alpha/2}.
\end{equation}
As a consequence, substituting $z=iy$ with $y \in [-\rho,\rho]$, and $\alpha=-1/2$,
$$
D_1(iy)^2\leq C_1 (ny)^{-1/2}(y+\sqrt{y^2+1})^{1/2}\leq 
C_1\left(\rho +\sqrt{\rho^2+1}\right)^{1/2} (ny)^{-1/2},
$$
and by the same argument with $\alpha=-\nu$,
$$
D_1(iy)^2\leq C_2 n^{1/2-\nu} y^{-\nu}(y+\sqrt{y^2+1})^{\nu}\leq 
C_2 \left(\rho+\sqrt{\rho^2+1}\right)^{\nu} n^{1/2-\nu} y^{-\nu}.
$$
The bound in the lemma follows for $y>0$ from these two estimates, for some constant $C_{\nu}$. 
Finally, from the definition of $D_1$, see \eqref{eq:D1(z)}, we have that if $y<0$, 
then $D_1(iy)=\overline{D_1(-iy)}$, so the  modulus is equal and the bound holds also in this case.
\end{proof}

Now we write together all the estimates computed before to obtain bounds for the functions 
$\eta_1$ and $\eta_2$ defined in \eqref{eta12z}. 
\begin{lemma}
\label{lem:bounds-eta1-eta2}
For $0<\nu\leq 1/2$, there exist constants $C_{\nu}, C'_{\nu} > 0$ such 
that for $n$ large enough and $y \in [0,\rho]$, we have the bounds
\begin{align} 
 |\eta_1(iy)| &\leq \left|j_1(iy) (D_1(iy) D_2(iy))^2 \right| 
\leq C_\nu \, y^\nu \,  e^{-2n\Re \varphi_-(iy)},  
\label{eq:bounds-eta1} \\
 |\eta_2(-iy)| &\leq \left|j_2(-iy) (D_1(-iy) D_2(-iy))^2 \right| 
\label{eq:bounds-eta2} 
	\leq C'_\nu  \, (n^{2\nu}y^\nu+ny^{1-\nu}) \, e^{-2n\Re \varphi_-(-iy)}.
\end{align}
\end{lemma}

\begin{proof} 
 We collect the results on $D_1$ (see formula \eqref{boundD1}), $D_2$ (we use the fact 
that this function does not depend on $n$ and formula  \eqref{D2at0}), $j_1$ and $j_2$ 
(formula \eqref{eq:formulas-j1-j2}). Then for some constant $C_{1,\nu}$ we simplify the bound to
$$
|\eta_1(iy)|\leq  C_{1,\nu} y^{\nu} \frac{1+(ny)^{1-2\nu}}{(1+(ny)^{1/2-\nu})^2}e^{-2n\Re \varphi_-(iy)}
\leq C_{\nu} y^{\nu} e^{-2n\Re \varphi_-(iy)}.
$$

Also,
$$
\begin{aligned}
|\eta_2(-iy)|
& \leq  C_{2,\nu} n^{2\nu} y^{\nu} \frac{1+(ny)^{1-2\nu}}{(1+(ny)^{1/2+\nu})(1+(ny)^{1/2-\nu})}
e^{-2n\Re \varphi_-(-iy)}\\
&\leq C'_{\nu} n^{2\nu} y^{\nu}(1+(ny)^{1-2\nu}) e^{-2n\Re \varphi_-(-iy)},
\end{aligned}
$$
and the result follows. 
\end{proof}

\subsubsection{Estimates for $\|K_1\|$ and $\|K_2\|$ as $n\to\infty$}

In order to estimate the norms of $K_1$ and $K_2$ we need the $\|\cdot\|_2$ norm 
of $\eta_1$ and $\eta_2$, see formula \eqref{eq:norms-K1-K2}. For this we use  the
estimate in Lemma \ref{lem:bounds-eta1-eta2} and the following bound on $\varphi(z)$:

\begin{lemma}
For every $s \in i \mathbb R$ we have 
\begin{equation} \label{estimate:Rephi} 
\begin{aligned}
	\Re \varphi_+(s) = \Re \varphi_-(s) 
		& = -|s| \log |s| + |s| \log(1+ \sqrt{1+s^2}) + \log(|s| + \sqrt{1+s^2}) \\
		& \geq |s| \log \frac{1}{|s|}.  
		\end{aligned}
\end{equation}
\end{lemma}

\begin{proof} We consider $\Re \varphi_-(s)$ with $s \in i \mathbb R_+$. The other cases follow by symmetry.
Let $x \in (0,1)$. Then by \eqref{gpgm} and \eqref{phig},
\[
 \varphi_{\pm}(x)=\pm \pi i\int_x^1 \psi(t)dt,
\] 
and so $\varphi'_+(x) = - \pi i \psi(x)$.
By analytic continuation we find
\[ \varphi'(z) = -\pi i \psi(z), \qquad \Re z > 0, \, \Im z > 0. \]
Then
\[ \varphi_-(s) = \varphi_+(x) + \int_x^s \varphi'(z) dz = 
	\varphi_+(x) - \pi i \int_x^s \psi(z) dz. \]
Since $\varphi_+(x)$ is purely imaginary, we obtain by taking the real part
and letting $x \to 0+$,
\[ \Re \varphi_-(s) =  \Im \pi  \int_0^s \psi(z) dz =  \Im \int_0^s \log \left( \frac{1+ (1-z^2)^{1/2}}{z} \right) dz, \]
	where we used \eqref{complexpsi} for $\psi$.
	The integral can be evaluated explicitly and it gives \eqref{estimate:Rephi}.
\end{proof}

Without loss of generality we assume in what follows that $\rho$ is small enough so that $|s|\log\frac{1}{|s|}>0$ for $s\in(-i\rho,i\rho)$.

In order to estimate integrals involving the functions $\varphi_{\pm}(z)$, we use \eqref{estimate:Rephi}, 
together with the following  technical lemma.
\begin{lemma}
\label{lem:bound-integral} For any $\alpha>-1$, there exists a constant $C = C_{\alpha}$ such 
that for $n$ large enough
\begin{equation} \label{bound-integral}
\int_0^{1/e} y^{\alpha}e^{-4n y\log\frac{1}{y}}dy\leq C (n\log n)^{-\alpha-1}. 
\end{equation}
\end{lemma}
\begin{proof} 
We split the integral into two parts and we estimate
\begin{align}
 \int_0^{1/e} y^{\alpha}e^{-4n y\log\frac{1}{y}} \, dy
&=\int_0^{1/\sqrt{n}} y^{\alpha}e^{-4n y\log\frac{1}{y}} \,dy
+\int_{1/\sqrt{n}}^{1/e} y^{\alpha}e^{-4n y\log\frac{1}{y}} \, dy \nonumber \\
&\leq \int_0^{1/\sqrt{n}} y^{\alpha}e^{-4 yn\log n} \, dy
+\int_{1/\sqrt{n}}^{1/e} y^{\alpha}e^{-2 \sqrt{n}\log{n}} \, dy. \label{eq:integrals_phi}
\end{align}
where for the second integral we used that $-y \log \frac{1}{y}$ is decreasing on $[0, \frac{1}{e}]$
and so $-y \log \frac{1}{y} \leq \frac{1}{\sqrt{n}} \log \sqrt{n}$ for $y \in [\frac{1}{\sqrt{n}}, \frac{1}{e}]$.
The first integral of \eqref{eq:integrals_phi} is estimated by extending the integral to $+\infty$
and the result is that it is $\mathcal{O}((n\log n)^{-\alpha-1})$ as $n \to \infty$.  
The second integral in \eqref{eq:integrals_phi} is $\mathcal{O}(e^{-c \sqrt{n}})$ as $n \to \infty$.
This gives the result.
\end{proof}

Combining the estimates in \eqref{eq:bounds-eta1}, \eqref{eq:bounds-eta2}, \eqref{estimate:Rephi} and \eqref{bound-integral} 
we obtain, whenever $2 \varepsilon < \frac{1}{e}$,
\begin{equation}\label{estimates_eta1}
\begin{aligned}
 \int_0^{2\varepsilon} |\eta_1(iy)|^2 dy&=\mathcal{O}(n^{-2\nu-1}(\log n)^{-2\nu-1}),\quad 
 \int_0^{2\varepsilon} \frac{|\eta_1(iy)|^2}{y} dy&=\mathcal{O}(n^{-2\nu}(\log n)^{-2\nu}),
\end{aligned}
 \end{equation}
and
\begin{equation}\label{estimates_eta2}
\begin{aligned}
 \int_0^{2\varepsilon} |\eta_2(-iy)|^2 dy&=\mathcal{O}(n^{2\nu-1}(\log n)^{-2\nu-1}),\quad
 \int_0^{2\varepsilon} \frac{|\eta_2(-iy)|^2}{y} dy&=\mathcal{O} (n^{2\nu}(\log n)^{-2\nu}),
\end{aligned}
\end{equation}
as $n \to \infty$.
To obtain \eqref{estimates_eta2} one has to consider the three different integrals coming 
from square of the factor $n^{2\nu}y^\nu+ny^{1-\nu}$ in \eqref{eq:bounds-eta1}--\eqref{eq:bounds-eta2},
 and retain the largest one. 

Hence, using \eqref{eq:norms-K1-K2} and \eqref{estimates_eta1}-\eqref{estimates_eta2} we have the bounds
\begin{equation}\label{normsK1K2}
\begin{aligned}
 \|K_1\|&\leq \left(\int_0^{2\varepsilon} \frac{|\eta_1(iy)|^2}{y} dy\right)^{1/2}=\mathcal{O}(n^{-\nu}(\log n)^{-\nu}),\\
 \|K_2\|&\leq \left(\int_0^{2\varepsilon} \frac{|\eta_2(-iy)|^2}{y} dy\right)^{1/2}=\mathcal{O}(n^{\nu}(\log n)^{-\nu}).
\end{aligned}
\end{equation}
Thus $K_1$ and $K_2$ are bounded operators between the Hilbert spaces $L^2([0,2i\varepsilon])$
and $L^2([-2i \varepsilon,0])$. In addition from \eqref{normsK1K2}, we get
\begin{equation}\label{estimate_K1K2}
 \|K_1 K_2\| \leq \| K_1 \| \, \| K_2 \| = \mathcal{O}((\log n)^{-2\nu}), \qquad n \to \infty,
\end{equation}
and similarly
\begin{equation} \label{estimate_K2K1}
	\|K_2 K_1\|=\mathcal{O}((\log n)^{-2\nu}), \qquad n\to\infty.
\end{equation}

\subsubsection{Proof of Lemma \ref{lem:Phat}}

\begin{proof}
It follows from \eqref{estimate_K1K2} and \eqref{estimate_K2K1} that the
operators $I-K_2K_1$ and $I-K_1K_2$ are invertible for $n$ large enough, and then
we can solve the equations \eqref{eq:formulas-f1-g1} and \eqref{eq:formulas-f2-g2}.
Thus we define the entries of the matrix $\widehat P$ as follows:
\begin{align}\label{entriesP:1}
\widehat P_{11} & = (I-K_2K_1)^{-1} 1, && \widehat P_{12}=K_1\widehat P_{11} \\
\widehat P_{21} & = K_2 \widehat P_{22},	&& \widehat P_{22}  =(I-K_1K_2)^{-1} 1.
\label{entriesP:2}
\end{align}
In \eqref{entriesP:1} and \eqref{entriesP:2} we use $1$ to denote the identically-one function 
in $L^2([0, 2 i \varepsilon])$ and $L^2([-2i \varepsilon,0])$, respectively.
Then \eqref{eq:formulas-f1-g1} and \eqref{eq:formulas-f2-g2} hold true, which means
that the equations in \eqref{hatP11P12} hold. This then also means that the jump condition 
\eqref{jumpcondition:Phat}  in the RH problem \ref{RHforPhat} is satisfied.

The equations  \eqref{hatP11P12} allow us to give estimates on $\widehat{P}(z)$. First of all
we obtain from \eqref{estimate_K1K2}-\eqref{estimate_K2K1}, \eqref{entriesP:1}, and \eqref{entriesP:2} that
\begin{equation} \label{normestimates1} 
	\| \widehat P_{11}\|_{L^2([0,2i\varepsilon])} = \mathcal{O}(1), \qquad  
		\| \widehat P_{22}\|_{L^2([-2i\varepsilon,0])} = \mathcal{O}(1), 
		\end{equation}
and then by \eqref{normsK1K2}
\begin{align} \label{normestimates2} 
	\| \widehat P_{12} \|_{L^2([-2i \varepsilon,0])} & \leq
	\| K_1 \| \, \| \widehat P_{11}\|_{L^2([0,2i\varepsilon])} 
		= \mathcal{O}(n^{-\nu}(\log n)^{-\nu}),\\
	\| \widehat P_{21} \|_{L^2([0, 2i \varepsilon])} & \leq
	\| K_2 \| \, \| \widehat P_{22}\|_{L^2([-2i\varepsilon,0])} 
		=\mathcal{O}(n^{\nu}(\log n)^{-\nu}).
	\end{align}

For pointwise estimates we use the distances
\[ d_+(z)=\dist (z,[0,2i\varepsilon]), \qquad d_-(z) =\dist(z,[-2i\varepsilon,0]). \]
Then by the first equation in  \eqref{hatP11P12}, we get for $z \in \mathbb C \setminus [-2i\varepsilon,0]$,
\begin{align*}
|\widehat{P}_{11}(z) - 1| & 
	\leq \frac{1}{2\pi d_-(z)} \left| \int_0^{2i \varepsilon} \eta_2(s) \widehat{P}_{12}(s) ds \right| 
		\leq \frac{1}{2\pi d_-(z)}  \| \eta_2\|_2 \, \| \widehat{P}_{12} \|_2 
		\end{align*}
where we used the Cauchy-Schwarz inequality, and $\| \cdot \|_2$ is the $L^2$ norm on $[-2i\varepsilon,0]$.
Thus by \eqref{estimates_eta2} and \eqref{normestimates2},
\begin{equation} \label{hatP11bound}
	 |\widehat{P}_{11}(z) - 1| = \frac{1}{d_-(z)} \, \mathcal{O} \left( n^{-1/2} (\log n)^{-2\nu-1/2}\right), 
\end{equation}
as $n \to \infty$, uniformly for $z \in \mathbb C \setminus [-2i\varepsilon,0]$. 
Using similar arguments, we obtain
\begin{align} \label{hatP12bound}
|\widehat{P}_{12}(z)| & = \frac{1}{d_+(z)} \mathcal{O} \left( n^{-\nu-1/2}(\log n)^{-\nu-1/2} \right),\\
\label{hatP21bound}
|\widehat{P}_{21}(z) | & =\frac{1}{d_-(z)} \mathcal{O} \left( n^{\nu-1/2} (\log n)^{-\nu-1/2} \right), \\
\label{hatP22bound}
|\widehat{P}_{22}(z) - 1| & = \frac{1}{d_+(z)} \mathcal{O} \left( n^{-1/2}  (\log n)^{-2\nu-1/2}\right),
 \end{align}
as $n \to \infty$, and the $\mathcal{O}$ terms are uniform in $z$. Observe that all $\mathcal{O}$ terms tend to $0$ as $n \to \infty$,
since $\nu \leq 1/2$.

It follows from \eqref{hatP11bound}--\eqref{hatP22bound} that $\widehat{P}(z) = I + \mathcal{O}(z^{-1})$ as $z \to \infty$
and therefore $\widehat{P}$ satisfies the RH problem \ref{RHforPhat}. 
For $|z| = 3 \varepsilon$ we have $d_{\pm}(z) \geq \varepsilon$. From \eqref{hatP11bound}--\eqref{hatP22bound} 
we then immediately find that the estimates in Lemma \ref{lem:Phat} hold, and the lemma is proved. 
\end{proof}

This also completes the proof of Proposition \ref{propo8}.

\subsection{Final transformation} \label{finaltrans}
Having $P$ as in Proposition \ref{propo8} we define the final transformation $Q \mapsto R$ as
\begin{equation}\label{R}
R(z)=\begin{cases}
Q(z), &  \textrm{for } |z| > 3 \varepsilon, \\
Q(z)P(z)^{-1}, & \textrm{for } |z| < 3 \varepsilon.
\end{cases}
\end{equation}
Recall that $Q$ is the solution of the RH problem \ref{RHforQ}.

Then $R$ has jumps on a contour $\Sigma_R$ that consists of $\Sigma_Q \setminus (-i \varepsilon, i \varepsilon)$
together with the circle of radius $3 \varepsilon$ around $0$, see Figure \ref{figR}.
Note that the jumps of $P$ and $Q$ coincide on $(-i\varepsilon, i \varepsilon)$, so that $R$ has
an analytic continuation across that interval.

\begin{figure}
\centerline{\includegraphics{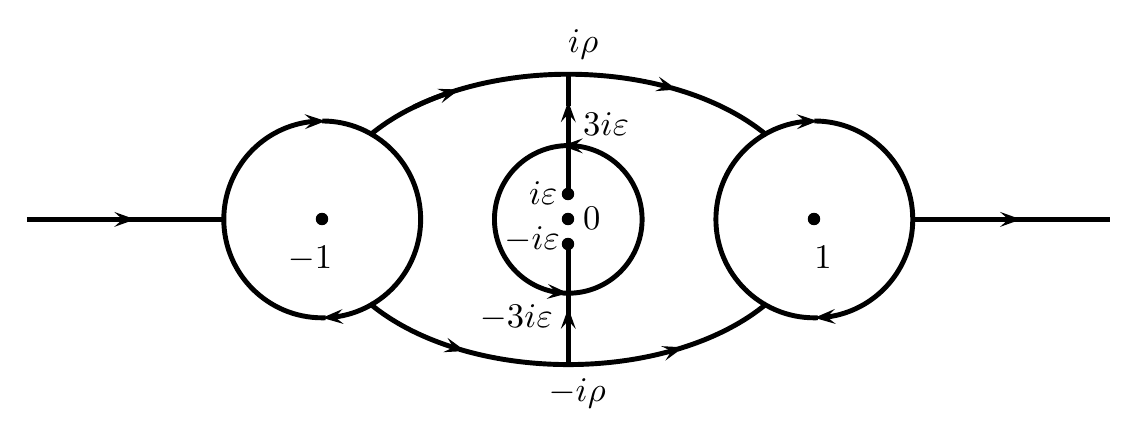}}
\caption{Contour $\Sigma_R$}
\label{figR}
\end{figure}

From RH problem \ref{RHforQ} and the definition \eqref{R} it follows that $R$ satisfies
the following RH problem.

\begin{rhp}\label{RHforR}
\begin{itemize}
\item[1)] $R : \mathbb C \setminus \Sigma_R \to \mathbb C^{2\times 2}$ is analytic.
\item[2)] $R$ satisfies the jump condition
$R_+ = R_- J_R$ on $\Sigma_R$ where
\begin{align} \label{jump:R}
J_R(z) = 
\begin{cases}
J_Q(z) & \text{ for } z \in \Sigma_R \text{ with } |z| > 3 \varepsilon, \\
P(z)^{-1} & \text{ for } |z| = 3 \varepsilon, \\
P_-(z) J_Q(z) P^{-1}_+(z) & \text{ for } z \in (-3i \varepsilon, - i \varepsilon) \cup
(i \varepsilon, 3i \varepsilon).
\end{cases}
\end{align}
\item[3)] As $z\rightarrow\infty$,
\begin{equation*}
R(z)=I+\mathcal{O}(1/z).
\end{equation*}
\end{itemize}
\end{rhp}

In order to solve this RH problem asymptotically for large $n$, we need to show that the jump matrices for $R(z)$ are 
close to the identity matrix uniformly for $z\in\Sigma_R$, see Figure \ref{figR}.

\begin{lemma}
The jump matrix $J_R$ in the RH problem for $R$ satisfies for some constant $c > 0$,
\begin{equation} \label{JRasymp}
	J_R(z) = \begin{cases} I + \mathcal{O}(\epsilon_n), &  \text{ for } |z| = 3 \varepsilon, \\
		I + \mathcal{O}(1/n), & \text{ for } |z\pm 1| = \delta, \\
		I + \mathcal{O}(e^{-cn}), & \text{ elsewhere on $\Sigma_R$},
		\end{cases}
		\end{equation}
as $n \to \infty$, where the $\mathcal{O}$ terms are uniform.
\end{lemma}
\begin{proof} 
For $z\in\Sigma_R$ with $|z|>3\varepsilon$, we have $J_R(z)=J_Q(z)$. On the boundary of the 
disks around the endpoints we have $J_Q(z)=I+\mathcal{O}(n^{-1})$, see \eqref{JQasymp1} 
and on the rest of $\Sigma_R$ except $(-i\rho,i\rho)$ we have $J_Q(z)=I+\mathcal{O}(e^{-cn})$ 
for some $c>0$, see \eqref{JQasymp2}. 

On the circle $|z|=3\varepsilon$, the jump is $J_R(z)=P(z)^{-1}$. We use \eqref{Phat} 
and the fact that $\widehat{P}(z)=I+\mathcal{O}(\epsilon_n)$, uniformly 
for $|z|=3\varepsilon$, to find that  
\[ J_R(z) = P(z)^{-1} = I+\mathcal{O}(\epsilon_n), \]
as given in \eqref{JRasymp}. 

For $z \in (3i\varepsilon,i\rho)$ we get from \eqref{jump:R} and \eqref{jumpQ1}
\begin{equation*}
 J_R(z)=J_Q(z)=D_{\infty}^{\sigma_3}N_0(z)\begin{pmatrix} 1 & 0\\ j_1(z)(D_1(z)D_2(z))^2 & 1\end{pmatrix}
 N_0^{-1}(z)D_{\infty}^{-\sigma_3}.
\end{equation*}
From \eqref{eq:bounds-eta1} and \eqref{estimate:Rephi}, we obtain for $y \in [0, \rho]$,
\begin{equation}\label{boundj1D1D2}
|j_1(iy)(D_1(iy)D_2(iy))^2|\leq C_{\nu} y^{\nu}e^{-2n y}, \qquad C_{\nu} > 0,
\end{equation} 
We also use \eqref{Dinftylimit}  and then \eqref{JRasymp} for $z \in (3i \varepsilon, i \rho)$
follows. The case $z \in (-i \rho, -3i \varepsilon)$ can be handled in a similar way.

What is left are the intervals $(i\varepsilon,3i\varepsilon)$ and $(-3i \varepsilon, -i \varepsilon)$.
For $z \in (i\varepsilon, 3i \varepsilon)$ we  find from \eqref{jump:R} and \eqref{Phat} 
that 
\begin{multline*}
 J_R(z)=D_{\infty}^{\sigma_3} N_0(z)
 \begin{pmatrix} 0 & -1 \\ 1 & 0\end{pmatrix}
 \widehat{P}_-(z) 
 \begin{pmatrix} 1 & -j_1(z)(D_1(z)D_2(z))^2\\ 0 & 1\end{pmatrix}\\
 \times \widehat{P}^{-1}_+(z)
 \begin{pmatrix} 0 & 1 \\ -1 & 0\end{pmatrix}
 N_0(z)^{-1}D_{\infty}^{-\sigma_3}.
\end{multline*}
Using \eqref{jump:Phat}-\eqref{jumpcondition:Phat} we rewrite this as
\begin{multline} \label{JRstep2}
J_R(z)  =I - j_1(z)(D_1(z)D_2(z))^2(1-\chi(z))
	D_{\infty}^{\sigma_3} N_0(z)
 \begin{pmatrix} 0 & -1 \\ 1 & 0\end{pmatrix}
 \widehat{P}_+(z) 
 \begin{pmatrix} 0 & 1 \\ 0 & 0 \end{pmatrix} \\
 \times \widehat{P}^{-1}_+(z)
 \begin{pmatrix} 0 & 1 \\ -1 & 0\end{pmatrix}
 N_0(z)^{-1}D_{\infty}^{-\sigma_3}.
\end{multline}
Here we note that $\det \widehat{P}(z) = 1$, which follows by standard arguments
from the RH problem \ref{RHforPhat}, and therefore 
$\widehat{P}^{-1}_+ = \begin{pmatrix} \widehat{P}_{22} & - \widehat{P}_{12} \\
	-\widehat{P}_{21} & \widehat{P}_{11} \end{pmatrix}_+$.
	Then a little calculation shows that \eqref{JRstep2} reduces to
\begin{align} \label{JRstep3}
J_R(z)  & =I + j_1(z)(D_1(z)D_2(z))^2(1-\chi(z)) \Lambda(z), \qquad z \in (i\varepsilon, 3 i \varepsilon),
\end{align}
where 
\begin{equation*}
\Lambda(z)=D_{\infty}^{\sigma_3}N_0(z)
\begin{pmatrix}
 -\widehat{P}_{11}(z)\widehat{P}_{21}(z)  & -\widehat{P}_{21}(z)^2 \\
 \widehat{P}_{11}(z)^2  & \widehat{P}_{11}(z)\widehat{P}_{21}(z) 
\end{pmatrix}
N_0^{-1}(z)D_{\infty}^{-\sigma_3}.
\end{equation*}
The functions $\widehat{P}_{11}$ and $\widehat{P}_{21}$ are analytic on $(i\varepsilon, 3 i \varepsilon)$
and so we do not have to take the $+$-boundary value. 

Then it follows from \eqref{Dinftylimit} and the estimates in \eqref{hatP11bound} and \eqref{hatP21bound} 
that all entries in $\Lambda$ are uniformly bounded as $n \to \infty$.
Then by \eqref{eq:bounds-eta1} and \eqref{JRstep3} we find \eqref{JRasymp} for $z \in (i \varepsilon, 3 i \varepsilon)$.
A similar argument shows that $J_R(z)$ is  exponentially close to the identity 
matrix for $z \in (-3i \varepsilon, -i \varepsilon)$ as well, and the lemma follows.
\end{proof}

As a consequence of \eqref{JRasymp}, the biggest estimates for $J_R - I$ are
on the circle $|z|=3\varepsilon$.  For $0<\nu\leq 1/2$, the jump matrix satisfies
(recall $\epsilon_n$ is given by \eqref{epsilonn})
\begin{equation}
 J_R(z) = I + \mathcal{O}(\epsilon_n), \qquad n \to \infty,
\end{equation}
uniformly for $z\in \Sigma_R$ where $\Sigma_R$ is the union of contours depicted in Figure \ref{figR}. 
Note that $J_R(z) \to I$ as $n \to \infty$, but the rate of convergence is remarkably slow.

Following standard arguments, we now find that for $n$ sufficiently large,
the RH problem \ref{RHforR} for $R$ is solvable, and 
\begin{equation}
\label{eq:asymptotics-R}
R(z) = I + \mathcal{O}(\epsilon_n), \qquad n\to\infty,
\end{equation}
uniformly for $z \in \mathbb C \setminus \Sigma_R$. 
The convergence rate in \eqref{eq:asymptotics-R} may not be optimal, since 
some of the bounds in the analysis may not be as sharp as possible. Note that for $\nu = 1/2$
we only have $R(z) = I + \mathcal{O}(\frac{1}{\log n})$, which is a very slow convergence.  

Since all of the transformations $X \mapsto U \mapsto T \mapsto S \mapsto Q \mapsto R$ are invertible,
we then also find that the RH problem for $X$ is solvable for $n$ large enough. In particular
we find that the polynomial $P_n = X_{11}$  exists for $n$ large enough.

\section{Proofs of the Theorems}\label{proofs}

\subsection{Proof of Theorem \ref{Th1}}  \label{section41} 

\begin{proof}
Following  the transformations of the Deift--Zhou steepest descent analysis and using formula
\eqref{eq:asymptotics-R}, we  obtain asymptotic information about $\widetilde{P}_n(z) = U_{11}(z)$ 
in the complex plane, see \eqref{UnX} and \eqref{tildePn}.  
Consider the region in Figure \ref{figR} which is outside the lens 
and outside of the disks around $z=\pm1$. In this case $U_{11}(z)=T_{11}(z)e^{ng(z)}$, and by \eqref{T},
\eqref{S}, \eqref{Q}, \eqref{R},
$$
T(z)=S(z)=Q(z)N(z)=R(z)N(z),
$$ which means that
\begin{equation}\label{proof:outer}
\begin{aligned}
\widetilde{P}_n(z) e^{-ng(z)} & = T_{11}(z) = R_{11}(z)N_{11}(z)+R_{12}(z)N_{21}(z)\\
&= N_{11}(z)(1+\mathcal{O}(\epsilon_n)) + N_{21}(z) \mathcal{O}(\epsilon_n),
\end{aligned}
\end{equation}
using \eqref{eq:asymptotics-R}. Here $\epsilon_n$ is given again by \eqref{epsilonn}.
We observe that $N_{11}=D_{\infty}N_{0,11} (D_1D_2)^{-1}$, from \eqref{solutionNgeneral}, and using 
\eqref{D1nlimit}, \eqref{Dinftylimit}, \eqref{eq:D2(z)} and \eqref{eq:N0} we get
\begin{equation}\label{N11asympexplicit}
N_{11}(z)=\left(\frac{z(z+(z^2-1)^{1/2})}{2(z^2-1)}\right)^{1/4}\left(\frac{(z^2-1)^{1/2}-i}{(z^2-1)^{1/2}+i}\right)^{-\nu/4}
\left(1+\mathcal{O} \left(\frac{\log n}{n}\right)\right), 
\end{equation}
as $n\to\infty$. Similarly, we also see that $N_{21}(z) = \mathcal{O}(1)$ as $n \to \infty$
and \eqref{asymp:Pn:outer} follows.

Since the lens can
be taken arbitrarily close to the interval $[-1,1]$ and the disks can be taken
arbitrarily small, the asymptotics \eqref{asymp:Pn:outer} is valid uniformly on
any compact subset of $\mathbb C \setminus [-1,1]$. This proves Theorem \ref{Th1}.
\end{proof}

\subsection{Proof of Theorem \ref{Th3}}   \label{section42} 

\begin{proof}
Inside the lens, but away from the endpoints and the origin, we use the relation 
\eqref{S} between the functions $T(z)$ and $S(z)$. Let $z$ be in the lens with $\Re z > 0$. Then we have
\begin{equation*}
T_{11}(z)=S_{11}(z) \pm S_{12}(z)\frac{e^{\frac{\nu \pi i}{2}-2n\varphi(z)}}{W_n(z)},
\end{equation*}
for $\pm \Im z > 0$, and therefore
\begin{equation*}
 \widetilde{P}_n(z) = e^{ng(z)} T_{11}(z)  =e^{ng(z)} \left[ S_{11}(z) \pm S_{12}(z)\frac{e^{\frac{\nu \pi i}{2}-2n\varphi(z)}}{W_n(z)} \right].
\end{equation*}
Since $S(z)=Q(z)N(z)$ away from the endpoints, and $Q(z)=R(z)$ away from the origin (if $|z|>3\varepsilon$),
see \eqref{Q} and \eqref{R}, we obtain
\begin{equation}\label{innerP}
\widetilde{P}_n(z) = e^{n g(z)} 
	\left[ N_{11}(z) \pm N_{12}(z)\frac{e^{\frac{\nu \pi i}{2}-2n\varphi(z)}}{W_n(z)}+\mathcal{O}(\epsilon_n) \right].
\end{equation}
for $\Re z \geq 0$, and $\pm \Im z > 0$.

We are going to simplify the expression \eqref{innerP} and we do it for $\Re z > 0$, $\Im z > 0$.
First we use \eqref{phifunction}, \eqref{Vfunction}, and \eqref{ell} in \eqref{innerP} to get 
\begin{equation}\label{innerP2}
\widetilde{P}_n(z) = \frac{e^{\frac{n \pi z}{2}}}{(2e)^n W_n(z)^{1/2}}
	\left[ N_{11}(z) W_n(z)^{1/2} e^{n \varphi(z)} + \frac{N_{12}(z)}{W_n(z)^{1/2}} e^{\frac{\nu \pi i}{2}-n\varphi(z)} 
		+\mathcal{O}(\epsilon_n) \right].
\end{equation}
From \eqref{solutionNgeneral} we have $N_{11} = D_{\infty} N_{0,11} (D_1 D_2)^{-1}$,
$N_{12} = D_{\infty} N_{0,12} D_1 D_2$ and so
\begin{multline}\label{innerP3}
\widetilde{P}_n(z) = \frac{D_{\infty} e^{\frac{n \pi z}{2}+\frac{\nu \pi i}{4}}}{(2e)^n W_n(z)^{1/2}} 
	\left[  \frac{N_{0,11}(z) W_n(z)^{1/2}}{D_1(z) D_2(z)} e^{-\frac{\nu \pi i}{4} + n \varphi(z)} \right. \\
	\left.  + 
		\frac{N_{0,12}(z) D_1(z) D_2(z)}{W_n(z)^{1/2}} e^{\frac{\nu \pi i}{4}-n\varphi(z)} 
		+\mathcal{O}(\epsilon_n) \right].
\end{multline}

Next we use \eqref{eq:N0} to write
\begin{equation*}
 N_{0,11}(z)= e^{-\frac{\pi i }{4}} \frac{f(z)^{1/2}}{\sqrt{2} (1-z^2)^{1/4}}, \qquad
 N_{0,12}(z)= e^{\frac{\pi i}{4}} \frac{f(z)^{-1/2}}{\sqrt{2} (1-z^2)^{1/4}}, 
\end{equation*}
where $(1-z^2)^{1/4}$ denotes the branch that is real and positive for $-1 < z < 1$ and $f(z)$ is given
 by \eqref{fz}. Thus
\begin{multline} \label{innerP4}
\widetilde{P}_n(z) = \frac{D_{\infty} 
e^{\frac{n \pi z}{2}+\frac{\nu \pi i}{4}}}{\sqrt{2} (2e)^n (1-z^2)^{1/4} W_n(z)^{1/2}} \\
\times
	\left[\left(  \frac{f(z)^{1/2} W_n(z)^{1/2}}{D_1(z) D_2(z)} e^{n \varphi(z)-\frac{\nu\pi i}{4}-
\frac{\pi i}{4}}  
		+  \frac{D_1(z) D_2(z)}{ f(z)^{1/2} W_n(z)^{1/2}} e^{-n\varphi(z)+\frac{\nu\pi i}{4}+
\frac{\pi i}{4}}\right)
 +\mathcal{O} \left(\epsilon_n \right)\right]. 
\end{multline}

The two terms in parenthesis are inverse of each other. We write all contributing factors 
in exponential form.  We have by
\eqref{phig}, \eqref{gpgm} and \eqref{D2andpsi}
\begin{align} \label{term1}	
	e^{n \varphi(z)} & = \exp(\pi i n \int_z^1 \psi(s) ds) \\
	D_2(z) e^{\frac{\nu \pi i}{4}} & = \exp \left(- \frac{\nu \pi}{2} \psi(z)\right)  \label{term2}
\end{align}
for $\Re z > 0$, $\Im z > 0$, and we note that by \eqref{Wnestimate} and \eqref{D1nlimit}
\begin{align}  \label{term3}
	\frac{W_n(z)^{1/2}}{D_1(z)} = f(z)^{-1/4} \left( 1 + \mathcal{O}\left( \frac{\log n}{n}\right) \right) \end{align}
as $n \to \infty$. Finally, we write
\begin{align} \label{term4}
	f(z)^{1/2}  = e^{\frac{i}{2} \arccos z}, \qquad \Im z > 0
\end{align}
and inserting \eqref{term1}--\eqref{term4} into \eqref{innerP4} we find
\eqref{asymp:Pn:inner}, where we also use  \eqref{Wnestimate}, \eqref{Dinftylimit} to simplify the first
factor.

A similar calculation leads to the same formula  \eqref{asymp:Pn:inner}
for $z \in E$ with $\Re z > 0$ and $\Im z < 0$.
\end{proof}

\subsection{Proof of Theorem \ref{Th0}}  \label{section43} 

\begin{proof}
It follows from \eqref{proof:outer} and \eqref{N11asympexplicit} 
that the leading factor in the outer asymptotics of 
$P_n(in\pi z)$ does not vanish for $z \in \mathbb C \setminus [-1,1]$. 

Let $\widetilde{P}_n(z) = (in\pi)^{-n} P_n(in\pi z)$ be the monic polynomial. Then
we find from \eqref{gfunction} that
\begin{equation} \label{eq:logPnconvergence} 
	\lim_{n\to\infty}\frac{1}{n} \log|  \widetilde{P}_n(z)| =  \Re g(z) = \int_{-1}^1 \log |z-x| \psi(x) dx,  
	\end{equation}
uniformly for $z$ in compact subsets of $\mathbb C \setminus [-1,1]$.
This implies that for any given compact subset $K \subset \overline{\mathbb{C}}\setminus[-1,1]$, 
the polynomial $\widetilde{P}_n$ does not have any zeros in $K$ for $n$ large enough.
In other words, all zeros of $\widetilde{P}_n$ tend to the interval $[-1,1]$ as $n \to \infty$.

In addition we find from \eqref{eq:logPnconvergence} that the zeros of $\widetilde{P}_n$
have $\psi(x)$ as limiting density. This follows from standard arguments in potential theory, 
see e.g.\ \cite{ST}.  
This proves Theorem \ref{Th0}.
\end{proof}

\subsection{Proof of Theorem \ref{Th2}}  \label{section44} 

Let $E$ be the neighborhood of $(-1,1)$ as in Theorem \ref{Th3}. 
Theorem \ref{Th2} will follow from  the asymptotic approximation \eqref{innerP} 
that is valid uniformly for $z$ in 
\[ E_{\delta} = E \setminus \left(D(-1, \delta) \cup D(0,\delta) \cup D(1, \delta)\right) \]
with $\Re z \geq 0$.

\begin{lemma} \label{lem:allzeros}
There is a constant $C > 0$ such that for large $n$ all zeros in $E_{\delta}$ satisfy
\begin{equation} \label{Allzeros} 
	\left| \Re \frac{\nu \pi}{2} \psi(z) - \Im \theta_n(z) \right| < C \epsilon_n. 
	\end{equation}
\end{lemma}
\begin{proof}
It is enough to consider $\Re z \geq 0$.

Let 
\[ F_n(z) = \exp \left( \frac{\nu \pi}{2} \psi(z) + i \theta_n(z)\right) \]
Then by \eqref{asymp:Pn:inner} we have that zeros of $\widetilde{P}_n$ in $E_{\delta}$
with $\Re z > 0$ are in the region where
\[ F_n(z) \left(1+ \mathcal{O}\left(\frac{\log n}{n}\right) \right) + F_n(z)^{-1}\left(1 + \mathcal{O}\left(\frac{\log n}{n}\right)\right) = \mathcal{O}(\epsilon_n). \]
This leads to
\[ F_n(z) + F_n(z)^{-1}= \mathcal{O}(\epsilon_n), \]
and so there is a constant $C > 0$ such that all zeros in $E_{\delta}$
satisfy 
\begin{equation} \label{Allzeros2} 
	|F_n(z) + F_n(z)^{-1}| \leq C \epsilon_n  \end{equation}
if $n$ is large enough.

Note that 
\[ |F_n(z)| =  \exp\left(\Re \frac{\nu \pi}{2} \psi(z) - \Im \theta_n(z) \right). \]
Thus if \eqref{Allzeros} is not satisfied then either
$ |F_n(z)| \geq \exp(C \epsilon_n)$ or $|F_n(z)| \leq \exp(-C \epsilon_n)$.
In both cases it follows that
\[ |F_n(z) + F_n(z)^{-1}| \geq e^{C \epsilon_n} - e^{-C \epsilon_n} \geq 2 C \epsilon_n. \]
Because of \eqref{Allzeros2} this cannot happen for zeros of $\widetilde{P}_n$ in $E_{\delta}$ if $n$ is large
enough, and the lemma follows.
\end{proof}

The lemma is the main ingredient to prove Theorem \ref{Th2}. 

\begin{proof}[Proof of Theorem \ref{Th2}]
In the proof we
use $c_1, c_2, \ldots$, to denote positive constants that do not depend
on $n$ or $z$. The constants will depend on $\delta > 0$.

It is easy to see from the definition \eqref{defthetan} 
that $\theta_n'(x)  \leq c_1 n < 0$ for $x \in (0, 1- \delta)$
This implies that for some constant $c_2 > 0$
\begin{equation} \label{zeros1} 
	\Im \theta_n(z) \begin{cases} \leq - c_2 n \Im z & \text{ for } z \in E_{\delta}, \Re z >0, \Im z \geq 0 \\
	\geq  c_2 n |\Im z| & \text{ for } z \in E_{\delta}, \Re z > 0, \Im z < 0 \end{cases} 
	\end{equation}
There are also constants $c_3, c_4 > 0$ such that
\begin{equation} \label{zeros2} 
	c_3 < \Re \frac{\nu \pi}{2} \psi(z) < c_4, \qquad z \in E_{\delta}, \Re z > 0, 
	\end{equation}
see \eqref{complexpsi}.
Thus if $\Im z \geq 0$ then by \eqref{zeros1} and \eqref{zeros2}
\[ \left| \Re \frac{\nu \pi}{2} \psi(z) - \Im \theta_n(z) \right| \geq c_2 n \Im z + c_3 \geq c_3 > 0\]
and thus there are no zeros in $E_{\delta}$ with $\Im z \geq 0$ by Lemma \ref{lem:allzeros} if $n$ is large enough.

For $\Im z \leq 0$ we have by \eqref{zeros1} and \eqref{zeros2}
\[ \left| \Re \frac{\nu \pi}{2} \psi(z) - \Im \theta_n(z) \right| \geq c_2 n |\Im z| - c_4  \]
It follows from this and Lemma \ref{lem:allzeros}  that for large $n$, there are 
no zeros with $\Im z \leq -\frac{c_5}{n}$ if $c_5 > c_4/c_2$. 

Now assume $z\in E_{\delta}$ with $ - \frac{c_5}{n} < \Im z < 0$ and $\Re z > 0$. Write $  z = x + i y$.
Then by Taylor expansion
\[ \frac{\nu \pi}{2} \psi(z) = \frac{\nu \pi}{2} \psi(x) + \mathcal{O}(1/n) \]
and, see also \eqref{defthetan},
\begin{align*} 
	\theta_n(z) & = \theta_n(x) + iy \theta_n'(x) + \mathcal{O}(1/n)  \\
		& = \theta_n(x) - iy n \pi \psi(x) + \mathcal{O}(1/n) 
		\end{align*}
and $\mathcal{O}$ terms are uniform for $z$ in the considered region.

Then since $\psi(x)$ and $\theta_n(x)$ are real, we have
\begin{align*} 
	\Re \frac{\nu \pi}{2} \psi(z) - \Im \theta_n(z) & = \frac{\nu \pi}{2} \psi(x) + yn \pi \psi(x) + \mathcal{O}(1/n) \\
		&= \left(\frac{\nu}{2} + ny \right) \pi \psi(x) + \mathcal{O}(1/n) 
		\end{align*}
Thus if $|\frac{\nu}{2 } + ny| \geq c_6 \epsilon_n$ then by the above and \eqref{zeros2} 
\[ \left| \Re \frac{\nu \pi}{2} \psi(z) - \Im \theta_n(z) \right|
	\geq \frac{2 c_6 c_3}{\nu} \epsilon_n  + \mathcal{O}(1/n) \]
and from Lemma \ref{lem:allzeros} it follows that $z = x + iy$ is not a zero if $c_6$ is large enough.

Thus for large $n$ all zeros $z = x + i y$ of $\widetilde{P}_n$ in $E_{\delta}$ satisfy
\[ \left|\frac{\nu}{2 } + n y\right| \leq c_6 \epsilon_n. \]
Then $in \pi z$ is a zero of $P_n$, see \eqref{tildePn}, and the
real part of this zero is $-n \pi y$ which differs from $\frac{\nu \pi}{2}$ by an amount
less than $\pi c_6 \epsilon_n$. This proves Theorem \ref{Th2}. 
\end{proof}

\section*{Acknowledgements}
We thank Daan Huybrechs for suggesting the problem and for stimulating conversations.

A. Dea\~{n}o gratefully acknowledges financial support from projects FWO G.0617.10 and FWO G.0641.11, 
funded by FWO (Fonds Wetenschappelijk Onderzoek, Research Fund Flanders, Belgium), and projects MTM2012--34787 and MTM2012-36732--C03--01, from Ministerio de Econom\'ia y Competitividad de Espa\~{n}a (Spanish Ministry of Economy and Competitivity).

A.B.J. Kuijlaars is supported by KU Leuven Research Grant OT/12/073, the Belgian Interuniversity Attraction Pole
P07/18, FWO Flanders projects G.0641.11 and G.0934.13, and by Grant No. MTM2011-28952-C02 of
the Spanish Ministry of Science and Innovation. 

P. Rom\'an was supported by the Coimbra Group Scholarships Programme at KULeuven in the
period February-May 2014.


\begin{thebibliography}{99}

\bibitem{AH}
A. Asheim and D. Huybrechs,
Complex Gaussian quadrature for oscillatory integral transforms. {\it IMA J Numer. Anal. 33(4) (2013), 1322--1341.} 

\bibitem{AMMT}
M. J. Atia, A. Mart\'inez-Finkelshtein, P. Mart\'inez-Gonz\'alez, and F. Thabet, 
Quadratic differentials and asymptotics of Laguerre polynomials with varying complex parameters, 
{\it J. Math. Anal. Appl. 416 (2014), 52--80.}

\bibitem{BB}
P. M. Bleher and T. Bothner, 
Exact solution of the six--vertex model with domain wall boundary conditions. 
Critical line between disordered and antiferroelectric phases. 
{\it Random Matrices: Theory Appl. 01, 1250012 (2012).}


\bibitem{Deano}
A. Dea\~no,
Large degree asymptotics of orthogonal polynomials with respect to an oscillatory weight on a bounded interval. 
preprint arXiv: 1402.2085

\bibitem{Deift}
P. A. Deift,
Orthogonal Polynomials and Random Matrices: a Riemann-Hilbert Approach. 
Courant Lecture Notes in Mathematics vol. 3. American Mathematical Society. Providence RI, 1999.

\bibitem{DIK}
P. Deift, A. Its, I. Krasovsky, Asymptotics of {T}oeplitz, {H}ankel, and {T}oeplitz+{H}ankel 
determinants with {F}isher--{H}artwig singularities. {\it Ann. of Math. 174 (2011), 1243--1299.} 


\bibitem{DIK2}
P. Deift, A. Its, I. Krasovsky, 
Toeplitz matrices and Toeplitz determinants under the impetus of the Ising model: some history and some recent results,
{\it Comm. Pure Appl. Math. 66 (2013),  1360--1438.}

\bibitem{DKMVZ}
P. Deift, T. Kriecherbauer, K.T-R McLaughlin, S. Venakides, and X. Zhou,
Uniform asymptotics for polynomials orthogonal with respect to varying exponential 
weights and applications to universality questions in random matrix theory,
Comm. Pure Appl. Math. 52 (1999),  1335--1425. 

\bibitem{FIK}
A. S. Fokas, A. R. Its, and A. V. Kitaev,
The isomonodromy approach to matrix models in 2D quantum gravity.
{\it Comm. Math. Phys. 147 (1992),  395--430.}


\bibitem{FMFS}
A. Foulqui\'{e} Moreno, A. Mart\'inez-Finkelshtein, and V. L. Sousa,
On a conjecture of A.~Magnus concerning the asymptotic behavior of the recurrence 
coefficients of the generalized Jacobi polynomials. 
{\it J. Approx. Theory 162 (2010), 807--831.}

\bibitem{FMFS2}
A. Foulqui\'e Moreno, A. Mart\'inez--Finkelshtein, V. L. Sousa,
Asymptotics of orthogonal polynomials for a weight with a jump on $[-1,1]$. 
{\it Constr. Approx. 33 (2011), 219--263.}


\bibitem{IK}
A. Its, I. Krasovsky, {H}ankel determinant and orthogonal polynomials for the 
{G}aussian weight with a jump. {\it Contemp. Math. 458 (2008), 215--247.}

\bibitem{KMcL}
T. Kriecherbauer and K. T.-R. McLaughlin, 
Strong asymptotics of polynomials orthogonal with respect to Freud weights. 
{\it Internat. Math. Res. Not. 6 (1999), 299--333.}

\bibitem{KuijMcL}
A. B. J. Kuijlaars and K. T.-R. McLaughlin,
Asymptotic zero behavior of Laguerre polynomial with negative parameter. 
{\it Constr. Approx. 20 (2004), 497--523.}

\bibitem{KMcVV}
A. B. J. Kuijlaars, K. T.-R. McLaughlin, W. Van Assche, and M. Vanlessen,
The Riemann-Hilbert approach to strong asymptotics for orthogonal polynomials on [-1,1].
{\it Adv. Math. 188 (2004), 337--398.}

\bibitem{KuijMF}
A. B. J. Kuijlaars and A. Mart\'inez-Finkelshtein,
Strong asymptotics for Jacobi polynomials with varying nonstandard parameters. 
{\it J. Anal. Math. 94 (2004), 195-234.}


\bibitem{DLMF}
NIST Digital Library of Mathematical Functions. 
http://dlmf.nist.gov/ Release 1.0.5 of 2012-10-01. 
Online companion to \cite{Olver:2010:NHMF}.

\bibitem{Olver:2010:NHMF} 
F.~W.~J. Olver, D.~W. Lozier, R.~F. Boisvert, and C.~W. Clark (eds.), 
NIST Handbook of Mathematical Functions. 
Cambridge University Press, New York, 2010. Print companion to \cite{DLMF}.

\bibitem{ST}
E. B. Saff and V. Totik,
Logarithmic Potentials with External Fields. Springer--Verlag. Berlin, 1997.

\bibitem{Szego}
G. Szeg\H{o}, Orthogonal Polynomials. American Mathematical Society. Providence RI, 1939.

\bibitem{Watson}
G. N. Watson, A Treatise on the Theory of Bessel Functions. Cambridge University Press. Cambridge, 1966.

\end{thebibliography}
\end{document}